\documentclass[11pt]{amsart}

\usepackage{geometry}
\usepackage{amsmath}
\usepackage{amssymb}
\usepackage{dsfont}
\usepackage{caption}
\usepackage{subcaption}
\usepackage{amsthm}
\usepackage{mathtools}
\usepackage{tikz}
\usepackage[british]{babel}
\usepackage{enumerate}
\tikzset{vtx/.style={inner sep=1.7pt, outer sep=0pt, circle, fill,draw}}
\usepackage{graphicx}
\usepackage{float}
\usepackage[sort=use]{glossaries}
\usepackage[hidelinks]{hyperref}

\title{Conflict-free Hypergraph Matchings and Coverings}
\author{Felix Joos}
\author{Dhruv Mubayi}
\author{Zak Smith}
\date{\today}

\thanks{The research leading to these results was partially supported by the Deutsche Forschungsgemeinschaft (DFG, German Research Foundation) -- 428212407 (F. Joos and Z. Smith).
}

\theoremstyle{definition}
\newtheorem{definition}{Definition}[section]

\theoremstyle{plain}
\newtheorem{theorem}[definition]{Theorem}
\newtheorem{lemma}[definition]{Lemma}

\newtheorem{observation}[definition]{Observation}

\numberwithin{equation}{section}

\newcommand{\cH}{\mathcal{H}}
\newcommand{\cM}{\mathcal{M}}

\graphicspath{ {./Figures/} }

\newcommand{\new}[1]{#1} 

\def\COMMENT#1{}

\makeglossaries

\newglossaryentry{nhood2} {
    name=\ensuremath{N_{\mathcal{C}}^{(2)}(e)},
    description={set of edges forming a conflict of size two with $ e $}
}

\newglossaryentry{codegree} {
    name=\ensuremath{d_{\mathcal{G}}(u, v)},
    description={codegree of vertices $ u $ and $ v $ in hypergraph $ \mathcal{G} $}
}

\newglossaryentry{maxdegree} {
    name=\ensuremath{\Delta_V(\mathcal{G})},
    description={maximum degree of all vertices in $ V \subseteq V(\mathcal{G}) $}
}

\newglossaryentry{uniformsubgraph} {
    name=\ensuremath{\mathcal{G}^{(j)}},
    description={subhypergraph of $ \mathcal{G} $ consisting of edges of size $ j $}
}

\newglossaryentry{mixedconflictunifsub} {
    name=\ensuremath{\mathcal{D}^{(j_1, j_2)}},
    description={conflicts consisting of $ j_1, j_2 $ edges from $ \mathcal{H}_1, \mathcal{H}_2 $, respectively}
}

\newglossaryentry{maxdegmixed} {
    name=\ensuremath{\Delta_{j'_1, j'_2}(\mathcal{D})},
    description={maximum codegree among edge sets consisting of $ j_1, j_2 $ edges from $ \mathcal{H}_1, \mathcal{H}_2 $, respectively}
}

\newglossaryentry{conflictsbyvertex} {
    name=\ensuremath{\mathcal{D}_x},
    description={conflicts containing vertex $ x $ in their $ \mathcal{H}_2 $-part}
}

\newglossaryentry{conflictsbyvertexpair} {
    name=\ensuremath{\mathcal{D}_{x, y}},
    description={conflicts containing both $ x $ and $ y $ in their $ \mathcal{H}_2 $-part}
}

\newglossaryentry{conflictfree} {
    name=$ \mathcal{C} $-free,
    description={contains no conflict from $ \mathcal{C} $}
}

\newglossaryentry{testfunction} {
    name=$ j $-uniform \new{$ \ell $}-test function,
    description={$ [0, \ell] $-valued function on sets of $ j $ edges which is non-zero only for matchings}
}

\newglossaryentry{link} {
    name=\ensuremath{\mathcal{H}_{-v}},
    description={link hypergraph (all sets whose union with $ v $ form an edge in $ \mathcal{H} $)}
}

\newglossaryentry{conflictpair} {
    name={$ (\varepsilon, \mathcal{C}) $-conflict-sharing pair},
    description={two edges with large link hypergraph in $ \mathcal{C} $}
}

\newglossaryentry{cbounded} {
    name={$ (d, \ell, \varepsilon) $-bounded},
    description={conflict hypergraph satisfying (C\ref{cond_c1})--(C\ref{cond_c5})}
}

\newglossaryentry{simplybounded} {
    name={$ (d, \ell, \varepsilon) $-simply-bounded},
    description={conflict hypergraph satisfying (D\ref{cond_mixed_conflictsize})--(D\ref{cond_mixed_codegree2})}
}

\newglossaryentry{trackable} {
    name={$ (d, \varepsilon, \mathcal{C}) $-trackable},
    description={test function satisfying (W\ref{cond_w1})--(W\ref{cond_w4})}
}

\newglossaryentry{semitrackable} {
    name={$ (d, \varepsilon, \mathcal{C}) $-\new{semi-}trackable},
    description={test function satisfying (W$\ref{cond_w1*}^*$), (W$\ref{cond_w2*}^*$), (W\ref{cond_w3}), (W\ref{cond_w4})}
}

\newglossaryentry{edgesetpvertices} {
    name=\ensuremath{V_P(E)},
    description={all vertices of $ P $ covered by $ \mathcal{H}_2 $-part of $ E $}
}

\newglossaryentry{unavoidability} {
    name=\ensuremath{A(E)},
    description={unavoidability of edge set $ E $}
}

\newglossaryentry{conflictsextension} {
    name=\ensuremath{\new{\mathcal{G}}_{[C, D]}},
    description={conflicts in $ \mathcal{G} $ containing $ C, D $ in their $ \mathcal{H}_1, \mathcal{H}_2 $-parts, respectively}
}

\newglossaryentry{maxdegreeunavoidability} {
    name=\ensuremath{\Delta_{j_1', j_2'}^A(\new{\mathcal{G}})},
    description={maximum codegree weighted by unavoidability}
}

\newglossaryentry{mixedbounded} {
    name={$ (d, \ell, \varepsilon, \delta) $-mixed-bounded},
    description={conflict hypergraph satisfying (E\ref{cond_mixed_general_conflictsize})--(E\ref{cond_mixed_general_codegree_j21})}
}

\newglossaryentry{unavoidabilityvertex} {
    name=\ensuremath{A(v)},
    description={unavoidability of vertex $ v $}
}

\newglossaryentry{constants} {
    name=\ensuremath{\ell', \Gamma, i^*},
    description={constants used in proof}
}

\newglossaryentry{hiparts} {
    name=\ensuremath{E_i(\mathcal{G})},
    description={$ \mathcal{H}_i $-parts of conflicts in $ \mathcal{G} $}
}

\newglossaryentry{conflictscontc} {
    name=\ensuremath{\mathcal{G}_{[C]}},
    description={conflicts of $ \mathcal{G} $ containing edge set $ C $ in their $ \mathcal{H}_1$-part}
}

\newglossaryentry{conflictscontcx} {
    name=\ensuremath{\mathcal{G}_{[C], x}},
    description={conflicts containing edge set $ C $ and vertex $ x $ in their $ \mathcal{H}_1, \mathcal{H}_2$-parts, respectively}
}

\newglossaryentry{conflictscontcxy} {
    name=\ensuremath{\mathcal{G}_{[C], x, y}},
    description={conflicts containing edge set $ C $ and vertices $ x, y $ in their $ \mathcal{H}_1, \mathcal{H}_2$-parts, respectively}
}

\newglossaryentry{testfnwx} {
    name=\ensuremath{w_x},
    description={test function used to count all potential conflicts containing $ x $}
}

\newglossaryentry{testfnwx'} {
    name=\ensuremath{w_x'},
    description={test function used to count blocked potential conflicts containing $ x $}
}

\newglossaryentry{testable} {
    name=testable,
    description={edge set which is a conflict-free matching with no conflict-sharing pairs}
}

\newglossaryentry{blockingedges} {
    name=\ensuremath{b_{C, y}},
    description={edges containing $ y $ which could block conflicts containing $ C $}
}

\newglossaryentry{alphax} {
    name=\ensuremath{\alpha_x},
    description={scaling constant for test functions $ w_x, w_x' $}
}

\newglossaryentry{nhoodxh2} {
    name=\ensuremath{N_x},
    description={edges in $ \mathcal{H}_2 $ containing vertex $ x $}
}

\newglossaryentry{degreexh2} {
    name=\ensuremath{d_x},
    description={size of $ N_x $}
}

\newglossaryentry{intersectionsencoding} {
    name=\ensuremath{\mathcal{P}(\overline{C})},
    description={encoding of all intersections of a collection $ \overline{C} $ of edge sets}
}

\newglossaryentry{allencodings} {
    name=\ensuremath{\mathcal{P}^*_{m, s, s'}},
    description={all possible $ \mathcal{P}(\overline{C}) $ for $ m $ sets with sum of sizes $ s' $ and union size $ s $}
}

\newglossaryentry{testfnwxbar} {
    name=\ensuremath{\overline{w}_x^{m, s, s'}},
    description={function used to count $ m $-sets of potential conflicts with common $ \mathcal{H}_2$-part of size 1}
}

\newglossaryentry{testfnwxmss} {
    name=\ensuremath{w_x^{m, s, s'}},
    description={restriction of $ \overline{w}_x^{m, s, s'} $ to testable sets}
}

\newglossaryentry{betax} {
    name=\ensuremath{\beta_x},
    description={scaling constant for test functions $ \overline{w}_x^{m, s, s'}, w_x^{m, s, s'} $}
}

\newglossaryentry{edgesetextensions} {
    name=\ensuremath{\mathcal{Z}_E^{(j)}},
    description={all edge sets of size $ j $ containing edge set $ E $}
}

\newglossaryentry{testfnw1} {
    name=\ensuremath{w_1},
    description={test function used to count vertices in $ P $ left uncovered by $ \mathcal{M}_1 $}
}

\newglossaryentry{safeedges} {
    name=\ensuremath{N_x^s},
    description={edges in $ \mathcal{H}_2 $ containing $ x $ which do not form a conflict with $ j_2 = 1 $}
}

\begin{document}

\begin{abstract}
	Recent work showing the existence of conflict-free almost-perfect hypergraph matchings has found many applications.
	We show that, assuming certain simple degree and codegree conditions on the hypergraph $ \mathcal{H} $ and the conflicts to be avoided, a conflict-free almost-perfect matching can be extended to one covering all vertices in a particular subset of $ V(\mathcal{H})$, by using an additional set of edges; 
    in particular, we ensure that our matching avoids all additional conflicts, which may consist of both old and new edges.
    
	This setup is useful for various applications in design theory and Ramsey theory. For example, our main result  provides a crucial tool in the recent proof of the high girth existence conjecture due to  Delcourt and Postle.  It also provides a black box which encapsulates many long and tedious calculations, greatly simplifying the proofs of results in generalised Ramsey theory. 
\end{abstract}

\maketitle

\section{Introduction} \label{section_introduction}
	Hypergraph matching problems can be used to model various central questions in combinatorics, and consequently have been studied for many years.
	Notably, Frankl and Rödl~\cite{frankl1985nearperfect}, as well as Pippenger~\cite{pippenger1989asymptotic}, proved that any $k$-uniform hypergraph ($k$ fixed) on $n$ vertices, in which each vertex belongs to roughly $d$ edges and any pair of \new{vertices} belongs to at most $o_d(d)$ edges, contains a matching covering $(1-o_d(1))n$ vertices.
	These theorems have undoubtedly had a vast number of applications.
	
	More recently, Delcourt and Postle~\cite{delcourt2022finding}, as well as Glock, Joos, Kim, Kühn, and Lichev~\cite{glock2022conflictfree}, generalised this result by introducing so-called conflict-free matchings.
	Here a conflict is a set of disjoint edges which is forbidden to be a subset of the matching.
	The main contribution in~\cite{delcourt2022finding,glock2022conflictfree} is as follows:
	under the same conditions as Frankl, Rödl, and Pippenger,
	and under sensible conditions on the set of conflicts,
	one can find an almost perfect matching that avoids all conflicts.
	This can again be applied to a number of problems, such as high girth decompositions, problems considered by Brown, Erd\H{o}s and S\'os~\cite{brown1973extremal} and various questions in generalised Ramsey theory~\cite{bal2024generalized, bennett2023edgecoloring, gomezleos2023new, joos2022ramsey, lane2024generalized}.
	
	The unfortunate drawback of these theorems is that they deal only with almost-perfect matchings,
	whereas in many applications it is desirable to obtain perfect matchings or at least almost perfect matchings that cover  a specified vertex subset entirely.
	Delcourt and Postle describe a setup in which this obstacle can be overcome in a particular setting.
	Specifically, they consider hypergraphs which are bipartite in the sense that the vertex set can be partitioned into two parts $A$ and $B$ with each edge containing exactly one vertex from $A$ (and they assume that the vertex degrees in $A$ are slightly higher than in $B$),
	in which case they find a conflict-free matching covering all vertices of $A$ (the matching is \textit{$A$-perfect}).
	The primary limitation of the bipartite hypergraphs in \cite{delcourt2022finding} is the fact that edges are only allowed to contain exactly one vertex from $A$, whereas in several applications this is not the case.

	Our main contribution is to obtain a stronger theorem in a  more general setup:
	we work with a `tripartite' hypergraph for which the vertex set has a partition into three sets $P,Q,R$, and we seek a $P$-perfect matching.
	The edge set is divided into two parts: edges containing $p$ vertices from $P$ and $q$ vertices from $Q$, and edges containing one vertex from $P$ and $r$ vertices from $R$.
	In this setup we prove the existence of a $ P $-perfect matching, avoiding conflicts which may consist of both types of edges (almost all vertices in $P$ are covered by edges of the first type and a few vertices in $P$ are covered by edges of the second type).
	\new{The proof} generalises a `two-stage method' for constructions in generalised Ramsey theory, first introduced by Bennett, Cushman, Dudek, and Prałat \cite{bennett2022erdhosgyarfas}, and subsequently used to prove various other results in this area; our theorem simplifies all such proofs significantly so that all technical computations are no longer needed, and mostly back-of-the-envelope calculations are sufficient.
 
	This framework also turns out to be useful in other settings, including high-girth coverings and designs; these applications are discussed further in Section~\ref{section_applications_matching}. We expect that there will be further applications in future.
	
\subsection{Tripartite Matching Theorem}
	\label{section_matching}
	
	\begin{figure}
		\begin{tikzpicture}
			\def \px {0};
			\def \qx {-3};
			\def \rx {3};			
			
			\node at (\px, 0) [draw,rectangle,minimum width=1cm,minimum height=4cm,rounded corners,thick,label=$P$] {};
			\node at (\qx, 0) [draw,rectangle,minimum width=1cm,minimum height=4cm,rounded corners,thick,label=$Q$] {};
			\node at (\rx, 0) [draw,rectangle,minimum width=1cm,minimum height=4cm,rounded corners,thick,label=$R$] {};
			
			\def \ylist {1.7, 1.4, 1.1, 0.8, 0.5, 0.2, -0.1, -0.4, -0.7, -1, -1.3, -1.6};
			\def \halfylistl {1.4, 0.8, 0.2, -0.4, -1, -1.6};
			\def \halfylists {1.4, 0.8, 0.2, -0.4, -1};
						
			\foreach \y in \halfylistl \node at (0, \y) [vtx,inner sep=1pt] {};
			
			\foreach \x in {\qx-0.2, \qx+0.2} \foreach \y in \ylist \node at (\x, \y) [vtx,inner sep=1pt] {};

			\foreach \x in {\rx-0.2,\rx+0.2} \foreach \y in \ylist \node at (\x, \y) [vtx,inner sep=1pt] {};			
			
			\def \edgesep {0.15};
			\def \pheight {0.6};
			\foreach \py in \halfylists {\foreach \qy in {\py + 0.3, \py, \py - 0.3} {
				\pgfmathrandominteger{\r}{3}{12}
				\draw[color=red, opacity=0.3, line cap=round, line join=round, line width=2pt] (\px+\edgesep, \py+\edgesep) -- (\px-\edgesep, \py+\edgesep) -- (\qx+0.1+\edgesep, \qy+\edgesep) -- (\qx-0.2-\edgesep, \qy+\edgesep) -- (\qx-0.2-\edgesep, \qy-0.3-\edgesep) -- (\qx+0.1+\edgesep, \qy-0.3-\edgesep) -- (\px-\edgesep, \py-\pheight-\edgesep) -- (\px+\edgesep, \py-\pheight-\edgesep) -- (\px+\edgesep, \py+\edgesep);
			}};

			\def \edgesep {0.15};
			\def \pheight {0};
			\foreach \py in \halfylistl {\foreach \qy in {\py + 0.3, \py} {
				\pgfmathrandominteger{\r}{3}{12}
				\draw[color=blue, opacity=0.3, line cap=round, line join=round, line width=2pt] (\px-\edgesep, \py+\edgesep) -- (\px+\edgesep, \py+\edgesep) -- (\rx-0.1-\edgesep, \qy+\edgesep) -- (\rx+0.2+\edgesep, \qy+\edgesep) -- (\rx+0.2+\edgesep, \qy-\edgesep) -- (\rx-0.1-\edgesep, \qy-\edgesep) -- (\px+\edgesep, \py-\pheight-\edgesep) -- (\px-\edgesep, \py-\pheight-\edgesep) -- (\px-\edgesep, \py+\edgesep);
			}};
			
			\def \edgesep {0.17};
			\def \py {1.4};
			\def \qy {1.7};
			\def \pheight {0.6};
			\draw[color=red, line cap=round, line join=round, line width=3pt] (\px+\edgesep, \py+\edgesep) -- (\px-\edgesep, \py+\edgesep) -- (\qx+0.1+\edgesep, \qy+\edgesep) -- (\qx-0.2-\edgesep, \qy+\edgesep) -- (\qx-0.2-\edgesep, \qy-0.3-\edgesep) -- (\qx+0.1+\edgesep, \qy-0.3-\edgesep) -- (\px-\edgesep, \py-\pheight-\edgesep) -- (\px+\edgesep, \py-\pheight-\edgesep) -- (\px+\edgesep, \py+\edgesep);
			\def \py {-0.4}
			\def \qy {-1.3};
			\draw[color=red, line cap=round, line join=round, line width=3pt] (\px+\edgesep, \py+\edgesep) -- (\px-\edgesep, \py+\edgesep) -- (\qx+0.1+\edgesep, \qy+\edgesep) -- (\qx-0.2-\edgesep, \qy+\edgesep) -- (\qx-0.2-\edgesep, \qy-0.3-\edgesep) -- (\qx+0.1+\edgesep, \qy-0.3-\edgesep) -- (\px-\edgesep, \py-\pheight-\edgesep) -- (\px+\edgesep, \py-\pheight-\edgesep) -- (\px+\edgesep, \py+\edgesep);
			
			\def \edgesep {0.17};
			\def \py {0.2};
			\def \qy {0.2};
			\def \pheight {0};
			\draw[color=blue, line cap=round, line join=round, line width=3pt] (\px-\edgesep, \py+\edgesep) -- (\px+\edgesep, \py+\edgesep) -- (\rx-0.1-\edgesep, \qy+\edgesep) -- (\rx+0.2+\edgesep, \qy+\edgesep) -- (\rx+0.2+\edgesep, \qy-\edgesep) -- (\rx-0.1-\edgesep, \qy-\edgesep) -- (\px+\edgesep, \py-\pheight-\edgesep) -- (\px-\edgesep, \py-\pheight-\edgesep) -- (\px-\edgesep, \py+\edgesep);
			\def \py {-1.6};
			\def \qy {-1.3};
			\draw[color=blue, line cap=round, line join=round, line width=3pt] (\px-\edgesep, \py+\edgesep) -- (\px+\edgesep, \py+\edgesep) -- (\rx-0.1-\edgesep, \qy+\edgesep) -- (\rx+0.2+\edgesep, \qy+\edgesep) -- (\rx+0.2+\edgesep, \qy-\edgesep) -- (\rx-0.1-\edgesep, \qy-\edgesep) -- (\px+\edgesep, \py-\pheight-\edgesep) -- (\px-\edgesep, \py-\pheight-\edgesep) -- (\px-\edgesep, \py+\edgesep);

			\node at (\qx - 1, 1.6) [red] {$\mathcal{M}_1$};
			\node at (\rx + 1, -1.3) [blue] {$\mathcal{M}_2$};
			
			\node at (0.5*\qx + 0.5*\px, -2.3) [red] {\small $\delta(\mathcal{H}_1) \approx \Delta(\mathcal{H}_1) \approx d$};
			\node at (\rx-1, -2.3) [blue] {\small $\Delta_R(\mathcal{H}_2) \le d^{\varepsilon^4} \delta_{P}(\mathcal{H}_2)$};
		\end{tikzpicture}
		\caption{Two matchings $ \mathcal{M}_1 \subseteq \mathcal{H}_1 $ and $ \mathcal{M}_2 \subseteq \mathcal{H}_2 $ in the hypergraph $ \mathcal{H} $, whose union forms a $ P $-perfect matching $ \mathcal{M} $; here $ p = 2, q = 4, r = 2 $.}
	\end{figure}
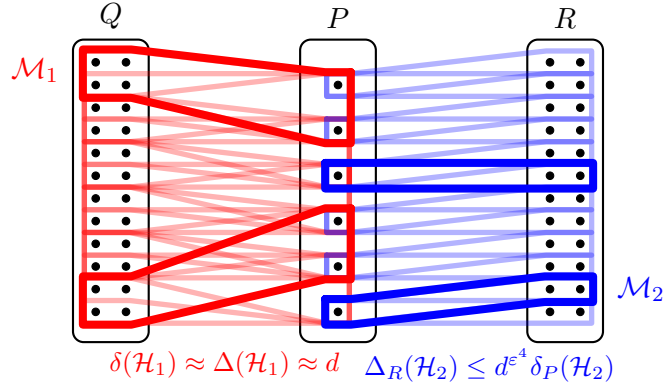


    Given two hypergraphs $ \mathcal{H} $ and $ \mathcal{C} $, say that $ \mathcal{C} $ is a \textit{conflict hypergraph for $ \mathcal{H} $} if $ V(\mathcal{C}) = E(\mathcal{H}) $, and in this case call the edges of $ \mathcal{C} $ \textit{conflicts}.
    \new{Say that a set of edges $ E \subseteq \mathcal{H} $ is \gls{conflictfree} if it does not contain any conflict from $ \mathcal{C} $.}
    Suppose that we are given the following setup:

    \begin{enumerate}[(S1)]
        \item \label{cond:setup_params} integers $ \ell \ge 2, d > 0, p \ge 1, q \ge 0, r \ge 1 $ with $ p + q \ge 2 $ and real $ \varepsilon > 0 $;
        \item \label{cond:setup_Psize} disjoint sets $ P, Q, R $ with $ d^{\varepsilon} \le |P| \le |P \cup Q| \le \exp(d^{\varepsilon^3}) $;
        \item hypergraph $ \mathcal{H}_1 $ whose edges consist of $ p $ vertices from $ P $ and $ q $ vertices from $ Q $;
        \item hypergraph $ \mathcal{H}_2 $ whose edges consist of a single vertex from $ P $ and $ r $ vertices from~$ R $;
        \item conflict hypergraph $ \mathcal{C} $ for $ \mathcal{H}_1 $;
        \item \label{cond:setup_mixed_conflicts} conflict hypergraph $ \mathcal{D} $ for $ \mathcal{H} \coloneqq \mathcal{H}_1 \cup \mathcal{H}_2 $.
    \end{enumerate}

    Assume that $ \mathcal{H} $ satisfies suitable degree conditions, and further that both $ \mathcal{C} $ and $ \mathcal{D} $ satisfy suitable boundedness conditions, all of which are specified in Section~\ref{section_conditions} in terms of $ d $ and $ \varepsilon $.
 
	\begin{theorem}
		\label{theorem_matching_simple}
		
		\new{Given $ p, q, r, \ell $ as above,} there exists $ \varepsilon_0 > 0 $ such that for all $ \varepsilon \in (0, \varepsilon_0) $, there exists $ d_0 $ such that given the above setup, the following holds for all $d \ge d_0$:  there exists a  $ P $-perfect \new{$ \mathcal{C} \cup \mathcal{D} $-free} matching $ \mathcal{M} \subseteq \mathcal{H} $. Furthermore, at most $ d^{-\varepsilon^4} |P| $ vertices of $ P $ belong to an edge in $ \mathcal{H}_2 \cap \mathcal{M} $.
	\end{theorem}
	
	The proof of Theorem~\ref{theorem_matching_simple}, given in Section~\ref{section_matching_proof}, consists of two stages, as mentioned: we first apply Theorem~\ref{theorem_cfhm_mod}, a variant of the original conflict-free hypergraph matching theorem \cite{glock2022conflictfree}, to the hypergraph $ \mathcal{H}_1 $ in order to obtain a \new{$ \mathcal{C} $}-free matching \new{$ \mathcal{M}_1 $} covering \new{all but at most a $ d^{-\varepsilon^3} $ fraction} of the vertices of $ P $ with edges from $ \mathcal{H}_1 $.
    \new{In the proof of Theorem~\ref{theorem_cfhm_mod} (which we apply as a black box), this matching is chosen randomly one edge at a time, and accordingly can be thought of as `pseudorandom' in an appropriate sense; specifically, we are able to show that certain weight functions on edge sets have roughly their `expected' value when summed over $ \mathcal{M}_1 $.}
	To extend \new{$\mathcal{M}_1$} to a $ P $-perfect matching, for each vertex $ x \in P $ which is not already covered, we randomly choose some edge from $ \mathcal{H}_2 $ to cover $ x $, and use the Lovász Local Lemma to show that with non-zero probability the resulting set \new{$ \mathcal{M}_2 $} of edges is indeed a matching, and the union $ \mathcal{M} $ of our two matchings is \new{$ \mathcal{D} $}-free.
    \new{In order to apply the local lemma, we use the aforementioned weight functions to bound the number of potential conflicts which could be introduced when choosing an edge from $ \mathcal{H}_2 $ to cover each $ x \in P $.}
	
\subsection{Notation}
	
	Write $ [i, n] = \{ i, \ldots, n \} $, so $ [n] = [1, n] $.
	Unless otherwise stated we identify hypergraphs with their edge sets, writing $ e \in \mathcal{G} $ to mean $ e \in E(\mathcal{G}) $.
    Given a set of vertices $ U \subseteq V(\mathcal{G}) $ in a hypergraph $ \mathcal{G} $, write $ d_{\mathcal{G}}(U) $ for the degree of $ U $ in $ \mathcal{G} $, that is the number of edges of $ \mathcal{G} $ containing $ U $; in the cases $ U = \{ u \} $ and $ U = \{ u, v \} $, where $ U $ consists of one or two vertices, we just write $ d_{\mathcal{G}}(u) $ and \gls{codegree} respectively.
	We omit the subscript if $ \mathcal{G} $ is obvious from context.
	Write $ \Delta_j(\mathcal{G}) $ for the maximum degree $ d_{\mathcal{G}}(U) $ among sets $ U \subseteq V(\mathcal{G}) $ of $ j $ vertices.
    Given a subset of the vertices $ V \subseteq V(\mathcal{G}) $, write \gls{maxdegree} for the maximum degree $ d_{\mathcal{G}}(u) $ of any single vertex $ u \in V $, and similarly $ \delta_V(\mathcal{G}) $ for the minimum degree; assume $ V = V(\mathcal{G}) $ if not specified.
	Given $ j \in \mathbb{N} $, write $ \gls{uniformsubgraph} \coloneqq \{ E \in \mathcal{G} : |E| = j \} $ for the subhypergraph of $ \mathcal{G} $ containing only those edges of size $ j $.

    We omit ceiling and floor symbols whenever they do not affect the argument.
    \new{We refer readers to the Glossary section at the end of the paper for a summary of more specific notation which appears throughout the paper.}

\subsection{This paper}

    \new{We start by outlining two significant applications of Theorem~\ref{theorem_matching_simple} in Section~\ref{section_applications_matching}.
    We then complete the formal statement in Section~\ref{section_conditions} by listing the various conditions required on our hypergraph and conflicts.
    In Section~\ref{section_proof_background}, we state our prerequisites, in particular Theorem~\ref{theorem_cfhm_mod}, a variant of the main theorem from~\cite{glock2022conflictfree}, as well as a more complicated but weaker set of conditions, for which we will in fact prove our main theorem.
    Finally, the proof itself is given in Section~\ref{section_matching_proof}, which begins with a more detailed outline.}
	
\section{Applications}
	\label{section_applications_matching}
	
	We briefly discuss here two applications of Theorem~\ref{theorem_matching_simple}.

\subsection{Conflict-free coverings and large girth designs}
	
	For a hypergraph $\cH$, we call a set of edges $\cM\subseteq \cH$ a \emph{covering} of $\cH$ if all vertices belong to some edge in $\cM$, and say it is \emph{perfect} if each vertex belongs to exactly one edge.
	In the setting of Frankl, Rödl and Pippenger, 
	the existence of almost perfect matchings and almost perfect coverings (coverings in which most vertices belong to exactly one edge) is equivalent
	and one object can easily be transformed into the other one.
	However, it is not obvious that just greedily adding edges to turn an almost perfect matching into an almost perfect covering can be done without introducing conflicts; Theorem~\ref{theorem_covering}, which follows easily from Theorem~\ref{theorem_matching_simple}, resolves this problem.
	It represents the natural analogue of Theorem~\ref{theorem_cfhm_simple}, providing a covering in place of a matching.
	
	\begin{theorem}
		\label{theorem_covering}
		Fix $ \ell, k \ge 2 $. There exists $ \varepsilon_0 $ such that for all $ \varepsilon \in (0, \varepsilon_0) $, there exists $ d_0 $ such that for all $ d \ge d_0 $ the following holds. Suppose $ \mathcal{H} $ is a $ k $-graph on \new{$ n \le \exp(d^{\varepsilon^3}) $} vertices. Assume that $ (1 - d^{-\varepsilon})d \le \delta(\mathcal{H}) \le \Delta(\mathcal{H}) \le d $ and $ \Delta_2(\mathcal{H}) \le d^{1 - \varepsilon} $. Let $ \mathcal{C} $ be a $ (d, \ell, \varepsilon) $-bounded conflict hypergraph for $ \mathcal{H} $. Then there is a $ \mathcal{C} $-free covering $ \mathcal{M} \subseteq \mathcal{H} $ such that all but $ d^{-\varepsilon^5}n $ vertices are covered exactly once, and no vertex is covered more than twice.
	\end{theorem}

	One of the original motivations for studying hypergraph matchings, and in particular conflict-free matchings, was the problem of finding almost-perfect Steiner systems of large girth.
	In general, \textit{a partial $ (m, s, t) $-Steiner system} is a collection $ \mathcal{S} $ of subsets of $ [m] $, each of size $ s $, such that every subset of $ [m] $ of size $ t $ is contained in at most one element of $ \mathcal{S} $; it is \textit{approximate} if it has size $ (1 - o(1)) {m \choose t} / {s \choose t} $.
	The \textit{girth} of $ \mathcal{S} $ is the smallest integer $ g \ge 2 $ such that some set of $ (s - t) g + t $ vertices induces at least $ g $ sets in $ \mathcal{S} $. 

    Recently, Delcourt and Postle~\cite{delcourt2024proof} proved the existence of perfect $ (m, s, t) $-Steiner systems of large girth, via a new \textit{refined absorption} method.
    Their result provides a common generalization of the existence conjecture for designs originating from the 1800s and Erd\H os’ conjecture from 1973 on the existence of high girth Steiner triple systems.
    Our main result, Theorem~\ref{theorem_matching_simple}, is a  crucial tool in their proof (Theorem 2.10 in~\cite{delcourt2024proof}).
    As noted earlier, our proof method uses the main result of~\cite{glock2022conflictfree}  which employs the random greedy process.
    On the other hand, Delcourt and Postle's proof (to appear in a forthcoming version of~\cite{delcourt2022finding}) of their Theorem 2.10 from~\cite{delcourt2024proof} uses the nibble method.
    A simpler version of Theorem~\ref{theorem_matching_simple}, without conflicts, is also used by Delcourt and Postle \cite{delcourt2024refinedabsorptionnewproof}, as well as by Delcourt, Postle, and Kelly \cite{delcourt2024cliquedecompositionsrandomgraphs, delcourt2024thresholdsnq2steinersystemsrefined} in other applications of their refined absorption method.
 
	Conflict-free hypergraph matchings can also be used to find approximate systems for a much more general class of quasirandom hypergraphs, where, for example, we restrict the choice of elements of $ \mathcal{S} $ to a randomly chosen subset of $ {[m] \choose s} $; see Theorem~1.4 of \cite{glock2022conflictfree}.
	The following analogous covering result can be easily deduced from our Theorem~\ref{theorem_covering}.
	Given real numbers $ a, b, c $, we write $ a = b \pm c $ to mean $ b - c \le a \le b + c $.
	
	\begin{theorem}
		\label{theorem_covering_application}
		For all $ c_0 > 0, \ell \ge 2 $ and $ s > t \ge 2 $, there exists $ \varepsilon_0 > 0 $ such that for all $ \varepsilon \in (0, \varepsilon_0) $, there exists $ m_0 $ such that the following holds for all $ m \ge m_0 $ and $ c \ge c_0 $.
		Let $ G $ be a $ t $-graph on $ m $ vertices and let $ \mathcal{K} $ be a collection of sets of size $ s $ which induce cliques in $ G $ such that any edge is contained in $ (1 \pm m^{-\varepsilon})c m^{s - t} $ elements of $ \mathcal{K} $.
		
		Then, there exists a subset $ \mathcal{S} \subseteq \mathcal{K} $ such that every edge of $ G $ is contained in at least one element of $ \mathcal{S} $, the proportion of edges of $ G $ contained in more than one element of $ \mathcal{S} $ is $ o_m(1) $, none are contained in more than two, and any subset of $ \mathcal{S} $ of size $ j $, where $ j \in [2, \ell] $, whose elements have pairwise intersections of size at most $ t - 1 $, spans more than $ (s - t)j + t $ points.
	\end{theorem}

	\subsection{Generalised Ramsey numbers}
	
	Given graphs $ G $ and $ H $, and $ q \in \mathbb{N} $, define the \textit{generalised Ramsey number $ r(G, H, q) $} to be the minimum number of colours needed to colour the edges of $ G $ in such a way that every copy of $ H $ receives at least $ q $ distinct colours.
	Bennett, Cushman, Dudek, and Prałat \cite{bennett2022erdhosgyarfas} showed that
	$ r(K_n, K_4, 5) = 5n/6 + o(n), $
	answering a question of Erd\H{o}s and Gy\'arf\'as, by introducing the aforementioned two-stage method in which they first colour most of the edges of $ K_n $ using a modified triangle removal process (requiring complicated technical analysis), and then complete this to a full colouring using the Lovász Local Lemma.
	Joos and Mubayi \cite{joos2022ramsey} simplified this method greatly by encoding the first stage as a suitable conflict-free hypergraph matching problem and applying the main theorem from \cite{glock2022conflictfree} as a black box, and demonstrated its versatility further by showing that $ r(K_n, C_4, 3) = n / 2 + o(n) $.
	This approach has subsequently been used to prove various similar and more general results \cite{bal2024generalized, bennett2023edgecoloring, gomezleos2023new, joos2022ramsey, lane2024generalized}.
	
	Our main theorem (Theorem~\ref{theorem_matching_simple}) formalises this two-stage method in a single statement, from which all of these colouring results follow; another main contribution of this paper is to consolidate all of the calculations required for the second stage, so that applications need only focus on constructing appropriate hypergraphs and conflicts satisfying our conditions.
	This massively simplifies the proofs of all existing results, since it now suffices to only check the orders of magnitude of the numbers of different types of conflicts.
    We illustrate this by providing in Appendix~\ref{section:appendix} a concise proof of the following result  which was stated (without proof) very recently by Bal, Bennett, Heath, and Zerbib \cite{bal2024generalized}.
    Given $ k \ge 2 $, write $ K_n^k $ for the complete $ k $-graph on $ n $ vertices, and $ C_{\ell}^k $ for the $ k $-uniform tight cycle of length $ \ell $; that is, edges $ e_1, \ldots, e_{\ell} $ on vertices $ v_1, \ldots, v_{\ell} $ such that $ e_i = \{ v_i, \ldots, v_{i + k - 1} \} $ (modulo $ \ell $) for each $ i \in [\ell] $.
    \begin{theorem} \label{theorem_cycles}
        For all $ k \ge 2 $ and $ \ell \ge k + 2 $, we have $ r(K_n^k, C_{\ell}^k, k + 1) \le n / (\ell - k) + o(n) $.
    \end{theorem}
We remark that the factor $\ell-k$ above is best possible assuming a well-known conjecture about the Tur\'an number of tight paths in hypergraphs.

\section{Conditions for the main theorem} \label{section_conditions}

    \new{We now specify the conditions on $ \mathcal{H}, \mathcal{C}$, and $ \mathcal{D} $ required for Theorem~\ref{theorem_matching_simple} to hold, thus completing its formal statement.
    Throughout this section, we work in the setup from Section~\ref{section_matching}, recalling (S\ref{cond:setup_params})--(S\ref{cond:setup_mixed_conflicts}).}

    \subsection{Degree conditions on $ \mathcal{H} $}
	
	We require that the hypergraph $ \mathcal{H} = \cH_1 \cup \cH_2 $ with $\cH_1,\cH_2\neq \emptyset$ satisfies the following conditions.

	\begin{enumerate}[(H1)]
		\item $ (1 - d^{-\varepsilon})d \le \delta_P(\mathcal{H}_1) \le \Delta(\mathcal{H}_1) \le d $; \label{cond_h_h1degree}
		\item $ \Delta_2(\mathcal{H}_1) \le d^{1 - \varepsilon} $; \label{cond_h_h1codegree}
		\item $\Delta_{R}(\mathcal{H}_2) \le d^{\varepsilon^4} \delta_{P}(\mathcal{H}_2) $; \label{cond_h_h2degree} 
		\item $ d(x, v) \le d^{-\varepsilon} \delta_{P}(\mathcal{H}_2) $ for each $ x \in P $ and $ v \in R $. \label{cond_h_h2codegree}
	\end{enumerate}
	
	This means that $ \mathcal{H}_1 $ is essentially regular for vertices in $ P $ and has small codegrees, although vertices in $ Q $ are allowed to have much lower (but not higher) degrees. Meanwhile in $ \mathcal{H}_2 $, every vertex in $ P $ must have degree at least a $ d^{-\varepsilon^4} $ proportion of the maximum degree in $ R $, and few edges in common with any particular vertex in $ R $. These conditions are used to ensure that we may choose an edge of $ \mathcal{H}_2 $ containing each $ x \in P $ such that the set of edges chosen do not overlap.
	
	\subsection{Boundedness conditions on $ \mathcal{C} $}
	
	The conflicts of $ \mathcal{C} $ consist only of edges from $ \mathcal{H}_1 $, and are avoided directly by Theorem~\ref{theorem_cfhm_mod}, so require the same boundedness conditions as in \cite{glock2022conflictfree}; the conditions we present here are not the most general possible, but suffice for most applications.
	For any edge $ e \in \mathcal{H}_1 $, we write $ \gls{nhood2} \coloneqq \{ f \in \mathcal{H}_1: \{ e, f \} \in \mathcal{C} \} $.
	Given $ \ell \ge 2 $ and $ d, \varepsilon > 0 $, say that \textit{$ \mathcal{C} $ is \gls{cbounded}} if
	\begin{enumerate}[(C1)]
		\item $ 2 \le |C| \le \ell $ for all $ C \in \mathcal{C} $; \label{cond_c1}
		\item $ \Delta(\mathcal{C}^{(j)}) \le \ell d^{j - 1} $ for all $ j \in [2, \ell] $; \label{cond_c2}
		\item $ \Delta_{j'}(\mathcal{C}^{(j)}) \le d^{j - j' - \varepsilon} $ for all $ j \in [2, \ell] $ and $ j' \in [2, j - 1] $; \label{cond_c3}
		\item $ |\{ f \in N_{\mathcal{C}}^{(2)}(e): v \in f \}| \le d^{1 - \varepsilon} $ for all $ e \in E(\mathcal{H}_1) $ and $ v \in V(\mathcal{H}_1) $; \label{cond_c4}
		\item $ |N_{\mathcal{C}}^{(2)}(e) \cap N_{\mathcal{C}}^{(2)}(f)| \le d^{1 - \varepsilon} $ for all disjoint $ e, f \in E(\mathcal{H}_1) $. \label{cond_c5}
	\end{enumerate}
	
	Note that, in many applications, all conflicts of $ \mathcal{C} $ have size at least 3, so conditions (C\ref{cond_c4}) and (C\ref{cond_c5}) are vacuously true.
    \new{For further intuition behind conditions (C\ref{cond_c1})--(C\ref{cond_c5}), we refer the reader to the proof overview (Section~1.1) in \cite{glock2022conflictfree}.}
	
	\subsection{Boundedness conditions on $ \mathcal{D} $} \label{section:simply_bounded}
	
	\new{Since it may be the case that the conflicts of $ \mathcal{D} $ consist only of edges from $ \mathcal{H}_2 $, or of two parts from $ \mathcal{H}_1 $ and $ \mathcal{H}_2 $, they must satisfy a new set of conditions to be avoided.}
	We extend our previous notation by writing \gls{mixedconflictunifsub} for the set of conflicts in $ \mathcal{D} $ consisting of $ j_1 $ edges from $ \mathcal{H}_1 $ and $ j_2 $ edges from $ \mathcal{H}_2 $.
	Similarly, we now write
    \new{\begin{equation} \label{eqn:defn_delta_j1j2}
        \gls{maxdegmixed} \coloneqq \max \left\{ d_{\mathcal{D}}(F_1 \cup F_2): F_1 \in {\mathcal{H}_1 \choose j_1'}, F_2 \in {\mathcal{H}_2 \choose j_2'} \right\}.
    \end{equation}}
	Given a vertex $ x \in P $, write also \gls{conflictsbyvertex} for the set of conflicts in $ \mathcal{D} $ containing $ x $ in their $ \mathcal{H}_2 $-part, and likewise \gls{conflictsbyvertexpair} for those containing both $ x $ and $ y $.
	Then say that \textit{$ \mathcal{D} $ is \gls{simplybounded}} if the following hold for all $ x, y \in P $, and $ j_1 \in [0, \ell], j_2 \in [2, \ell] $.
	\begin{enumerate}[(D1)]
		\item $ 2 \le |D \cap \mathcal{H}_2| \le |D| \le \ell $ for each conflict $ D \in \mathcal{D} $; \label{cond_mixed_conflictsize}
		\item $ |\mathcal{D}^{(j_1, j_2)}_x| \le d^{j_1 + \varepsilon^4} \delta_{P}(\mathcal{H}_2)^{j_2} $; \label{cond_mixed_degree}
		\item $ \Delta_{j', 0}(\mathcal{D}^{(j_1, j_2)}_x) \le d^{j_1 - j' - \varepsilon} \delta_{P}(\mathcal{H}_2)^{j_2} $ for each $ j' \in [j_1] $; \label{cond_mixed_codegree1}
		\item $ |\mathcal{D}^{(j_1, j_2)}_{x, y}| \le d^{j_1 - \varepsilon} \delta_{P}(\mathcal{H}_2)^{j_2} $. \label{cond_mixed_codegree2}
	\end{enumerate}
	
    \new{Referring to the brief proof outline in Section~\ref{section_matching}, the bound in (D\ref{cond_mixed_degree}) can be understood heuristically as follows.
    For $ D \in \mathcal{D}_x $ to be contained in $ \mathcal{M} $, all edges of $ D \cap \mathcal{H}_1 $ must be chosen in $ \mathcal{M}_1 $ and all edges of $ D \cap \mathcal{H}_2 $ must be chosen in $ \mathcal{M}_2 $; note in particular that, for the latter to be possible, all vertices $ y \in P \cap \bigcup (D \cap \mathcal{H}_2) $ must be left uncovered by $ \mathcal{M}_1 $.
    Under the heuristic that edges are chosen roughly uniformly and independently by Theorem~\ref{theorem_cfhm_mod}, the probability of the $ j_1 $ edges of $ D \cap \mathcal{H}_1 $ all being contained in $ \mathcal{M}_1 $ is roughly $ d^{-j_1} $.
    Under the heuristic that $ \mathcal{M}_1 $ covers vertices of $ P $ roughly uniformly, the probability that even one vertex $ y \in P \cap \bigcup (D \cap \mathcal{H}_2) \setminus \{ x \} $ is left uncovered is roughly $ d^{-\varepsilon^3} $.
    Therefore, after choosing $ \mathcal{M}_1 $, we expect there to be at most $ d^{\varepsilon^4 - \varepsilon^3} \delta_P(\mathcal{H}_2)^{j_2} $ conflicts $ D \in \mathcal{D}_x $ which pose a threat when choosing $ \mathcal{M}_2 $.
    Ignoring issues of dependence, each of the remaining $ j_2 $ edges of $ D $ are then included in $ \mathcal{M}_2 $ with probability at most $ \delta_P(\mathcal{H}_2)^{-1} $, so in total we expect at most $ d^{\varepsilon^4 - \varepsilon^3} = o(1) $ conflicts from $ \mathcal{D}_x $ to be contained in $ \mathcal{M} $.
    Intuitively, (D\ref{cond_mixed_codegree1}) and (D\ref{cond_mixed_codegree2}) behave similarly to (C\ref{cond_c3}), to ensure that conflicts are `well spread-out' in the hypergraph, which avoids the number of potential conflicts being dominated by rare events with large effects.}

    \new{We may now formally re-state Theorem~\ref{theorem_matching_simple}.}

    \begin{theorem}
		\label{theorem_matching}
		
		\new{Given $ p, q, r, \ell $ satisfying (S\ref{cond:setup_params}), there exists $ \varepsilon_0 > 0 $ such that for all $ \varepsilon \in (0, \varepsilon_0) $, there exists $ d_0 $ such that, for all $d \ge d_0$, given $ \mathcal{H}, \mathcal{C}, \mathcal{D} $ satisfying (S\ref{cond:setup_Psize})--(S\ref{cond:setup_mixed_conflicts}), the following holds.
        Assume that $ \mathcal{H} $ satisfies (H\ref{cond_h_h1degree})--(H\ref{cond_h_h2codegree}), $ \mathcal{C} $ is $ (d, \ell, \varepsilon) $-bounded, and $ \mathcal{D} $ is $ (d, \ell, \varepsilon) $-simply-bounded.
        Then there exists a  $ P $-perfect $ \mathcal{C} \cup \mathcal{D} $-free matching $ \mathcal{M} \subseteq \mathcal{H} $.
        Furthermore, at most $ d^{-\varepsilon^4} |P| $ vertices of $ P $ belong to an edge in $ \mathcal{H}_2 \cap \mathcal{M} $.}
	\end{theorem}

    We actually prove Theorem~\ref{theorem_matching} for a slightly more general set of conditions, given in Section~\ref{section_general_conditions}, which in particular allow for conflicts \new{$ D $ with $ j_2 = |D \cap \mathcal{H}_2| = 1 $}; in many applications, however, this type of conflict does not occur.

\section{Preparation for the Proof}
	\label{section_proof_background}

	In this section, we begin by deducing our required variant of the conflict-free hypergraph matchings theorem of \cite{glock2022conflictfree}, as well as giving the more general set of conditions under which we prove Theorem~\ref{theorem_matching}.
	We then proceed with the proof itself in Section~\ref{section_matching_proof}.
		
	\subsection{The Conflict-Free Hypergraph Matchings Theorem}
	\label{section_proof_background_cfhm}
			
		To state the theorem we need, we must first make a further definition; assume $ \ell $ is a given integer, as in Section~\ref{section_matching}.
        \new{Recalling that we write $ \mathcal{H} $ to refer to $ E(\mathcal{H}) $, and} given $ j \in \mathbb{N} \cup \{ 0 \} $,\footnote{Note that in the original paper $ j $ is assumed to be strictly positive, but the conclusion of the theorem still holds in the trivial case $ j = 0 $. This will simplify our notation later.}
		say that $ w\colon\ {\mathcal{H} \choose j} \to [0, \ell] $ is a \gls{testfunction} for $ \mathcal{H} $ if $ w(E) = 0 $ whenever $ E \in {\mathcal{H} \choose j} $ is not a matching.
		Write $ w(X) \coloneqq \sum_{x \in X} w(x) $ for  $ X \subseteq {\mathcal{H} \choose j} $, and $ w(E) \coloneqq w({E \choose j}) $ for general $ E \subseteq \mathcal{H} $.
		
		Given a vertex $ v $ in a hypergraph $ \mathcal{H} $, the \textit{link of $ v $ in $ \mathcal{H} $} is the hypergraph $ \gls{link} = \{ E \setminus \{ v \} : E \in \mathcal{H}, v \in E \} $ on vertex set $ V(\mathcal{H}) \setminus \{ v \} $, that is the set of all \textit{partial} edges which are completed by $ v $ to form an edge of $ \mathcal{H} $.
		Now say that a pair of edges $ e, f \in \mathcal{H} $ is an \textit{\gls{conflictpair}} (or just \textit{$ \varepsilon $-conflict-sharing pair} since $ \mathcal{C} $ is usually obvious from context) if $ |(\mathcal{C}_{-e})^{(j')} \cap (\mathcal{C}_{-f})^{(j')}| > d^{j' - \varepsilon} $ for some $ j' \in [\ell - 1] $.
		Then given a conflict hypergraph $ \mathcal{C} $ for $ \mathcal{\mathcal{H}} $, \new{a $ j $-uniform $ \ell $-test function $ w $ for $ \mathcal{H} $,} and $ d, \varepsilon > 0 $, we say that \textit{$ w $ is \gls{trackable}} if
		\begin{enumerate}[(W1)]
			\item $ w(\mathcal{H}) \ge d^{j + \varepsilon} $; \label{cond_w1}
			\item $ w(\{ E \in {\mathcal{H} \choose j} : E \supseteq E' \}) \le w(\mathcal{H}) / d^{j' + \varepsilon / 2} $ for all $ j' \in [j - 1] $ and $ E' \in {\mathcal{H} \choose j'} $; \label{cond_w2}
			\item \new{$ w(E) = 0 $ for any $ E \in {\mathcal{H} \choose j} $ with $ e,f \in E $ for some} $ 2 \varepsilon $-conflict-sharing pair $ e, f \in \mathcal{H} $; \label{cond_w3}
			\item $ w(E) = 0 $ for all $ E \in {\mathcal{H} \choose j} $ which are not $ \mathcal{C} $-free. \label{cond_w4}
		\end{enumerate}

        \new{We will make use of trackable test functions in our proof to bound the number of mixed conflicts whose $ \mathcal{H}_1 $-part is chosen in the first matching, meaning that their $ \mathcal{H}_2 $-part needs to be avoided in the second matching.
        For some intuition behind conditions (W\ref{cond_w1})--(W\ref{cond_w4}), we refer the reader again to \cite{glock2022conflictfree}, specifically the explanation of their conditions (Z1)--(Z4) in Section~3.}
		We may now state the original theorem, a simplified variant of the main result from \cite{glock2022conflictfree}.
		Note that the absence of conflict-sharing pairs in the matching is not given by any of the theorem statements in the original paper, but follows from the proof, in which such pairs are added as conflicts of size 2.\footnote{See Lemmas 8.5 and 8.6 in \cite{glock2022conflictfree}, which refer to conflict-sharing pairs as \textit{bad pairs}; note that while the statements given there only avoid $ \varepsilon / 4 $-conflict-sharing pairs, this can easily be improved to $ \varepsilon / 2 $ by being more conservative with $ \varepsilon $, provided that all test functions are zero on sets containing an $ \varepsilon / 2 $-conflict-sharing pair; this follows from (W\ref{cond_w3}), noting that our condition (W\ref{cond_w3}) uses $ 2 \varepsilon $ in place of $ \varepsilon $, unlike in the original paper.
		Similarly, taking such extra care allows us to require only that test functions are $ (d, \varepsilon / 4, \mathcal{C}) $-trackable, rather than with $ \varepsilon $ as in the original statement, and to replace $ \varepsilon $ by $ \varepsilon / 2 $ in condition (W\ref{cond_w2}).}
		
		\begin{theorem}
			\label{theorem_cfhm_simple}
			
			For all $ k, \ell \ge 2 $, there exists $ \varepsilon_0 > 0 $ such that for all $ \varepsilon \in (0, \varepsilon_0) $, there exists $ d_0 $ such that the following holds for all $ d \ge d_0 $.
			Suppose $ \mathcal{H} $ is a $ k $-graph on $ n \le \new{2} \exp(d^{\varepsilon^3}) $ vertices with $ (1 - d^{-\varepsilon}) d \le \delta(\mathcal{H}) \le \Delta(\mathcal{H}) \le d $ and $ \Delta_2(\mathcal{H}) \le d^{1 - \varepsilon} $ and suppose $ \mathcal{C} $ is a $ (d, \ell, \varepsilon) $-bounded conflict hypergraph for $ \mathcal{H} $.
			Suppose also that $ \mathcal{Y} $ is a set of $ (d, \varepsilon / \new{10}, \mathcal{C}) $-trackable \new{$ \ell $-}test functions for $ \mathcal{H} $ of uniformity at most $ \ell $ with $ |\mathcal{Y}| \le \exp(d^{\varepsilon^3}) $.
			Then, there exists a $ \mathcal{C} $-free matching $ \mathcal{M} \subseteq \mathcal{H} $ of size at least $ (1 - d^{-2\varepsilon^3})\frac{n}{k} $ with $ w(\mathcal{M}) = (1 \pm d^{-\varepsilon^3}) d^{-j} w(\mathcal{H}) $ for all $ j $-uniform $ w \in \mathcal{Y} $.
			Furthermore, $ \mathcal{M} $ contains no $ (\varepsilon / 2, \mathcal{C}) $-conflict-sharing pairs.
		\end{theorem}
		
		In order to prove Theorem~\ref{theorem_matching}, we require a slight extension, for which we make another definition; we say a $ j $-uniform test function \textit{$ w $ is \gls{semitrackable}} if it satisfies the alternative conditions
		\begin{enumerate}[(W$1^*$)]
			\item $ w(\mathcal{H}) \le d^{j + 2 \varepsilon} $; \label{cond_w1*}
			\item $ w(\{ E \in {\mathcal{H} \choose j} : E \supseteq E' \}) \le \varepsilon^{-1} d^{j - j'} $ for all $ j' \in [j - 1] $ and $ E' \in {\mathcal{H} \choose j'} $; \label{cond_w2*}
		\end{enumerate}		
		as well as (W\ref{cond_w3}) and (W\ref{cond_w4}).
        Observe also here that, if (W\ref{cond_w1}) holds, then (W$\ref{cond_w2*}^*$) implies (W\ref{cond_w2}), so we will usually just check (W$\ref{cond_w2*}^*$) when proving that test functions are trackable.
        
		For such functions, it is not in general possible to guarantee the same estimate \new{$ w(\mathcal{M}) = (1 \pm d^{-\varepsilon^3}) d^{-j} w(\mathcal{H}) $}, because the heuristically expected value is too small and therefore subject to outlying events; for a simple example, let $ v \in V(\mathcal{H})$, choose $ Z $ to be some set of size $ d(v) / 2 $ of edges that contain $ v $, and take $ w $ as the indicator function of $ Z $, then \new{(heuristically, imagining $ \mathcal{M} $ to be appropriately pseudorandom)}, the expected value of $ w(\mathcal{M}) $ is \new{(approximately)} $ 1 / 2 $ but clearly we cannot guarantee concentration close to $ 1 / 2 $ since $ w $ takes values in $ \{ 0, 1 \} $.
		It is however obvious in this case that $ w(\mathcal{M}) $ can be bounded from above, namely by $ 1 $.
		More generally, we might hope that for \new{semi-}trackable test functions, we can make use of the bound on $ w(\mathcal{H}) $ to guarantee that at least $ w(\mathcal{M}) $ is not too large; indeed, this turns out to be possible.
		
		\begin{theorem}
			\label{theorem_cfhm_mod}
			
			Assume the setup of Theorem~\ref{theorem_cfhm_simple}, but allow $ \mathcal{Y} $ to contain also some $ (d, \varepsilon / \new{10}, \mathcal{C}) $-\new{semi-}trackable \new{$ \ell $-}test functions.
			Then there exists a $ \mathcal{C} $-free matching $ \mathcal{M} \subseteq \mathcal{H} $ of size at least $ (1 - d^{-\varepsilon^3})\frac{n}{k} $, containing no $ (\varepsilon / 2, \mathcal{C}) $-conflict-sharing pairs, with $ w(\mathcal{M}) = (1 \pm d^{-\varepsilon^3}) d^{-j} w(\mathcal{H}) $ for all $ j $-uniform trackable $ w \in \mathcal{Y} $ and $ w(\mathcal{M}) \le d^{\varepsilon / \new{4}} $ for all \new{semi-}trackable $ w \in \mathcal{Y} $.
		\end{theorem}
		
		\begin{proof}
			Our strategy will be to extend $ \mathcal{H} $ using some new dummy vertices and edges, allowing us to extend $ w $ to a new test function $ w' $ which agrees with $ w $ on $ \mathcal{H} $, but is also positive on sufficiently many subsets of the new edges to satisfy (W\ref{cond_w1}).
			We then obtain the usual estimate for the value of $ w'(\mathcal{M}) $, which in particular gives us the required crude upper bound for $ w(\mathcal{M}) $.
		
			Let $ S $ be a set of $ m $ new vertices, disjoint from $ V(\mathcal{H}) $, with $ m $ chosen such that $ (1 - d^{-\varepsilon}) d \le {m - 1 \choose k - 1} \le d $\new{, noting that this is possible for $ \varepsilon < (k - 1)^{-1} $.}
			Note\COMMENT{Indeed, $ (\frac{m - 1}{k - 1})^{k - 1} \le {m - 1 \choose k - 1} \le d $ so $ m \le 1 + (k - 1) d^{1 / (k - 1)} \le k d^{1/(k - 1)} $. Conversely $ (m - 1)^{k - 1} \ge {m - 1 \choose k - 1} $ so $ m \ge 1 + ((1 - d^{-\varepsilon})d)^{1 / (k - 1)} \ge d^{\varepsilon} $ if $ \varepsilon_0 < (k - 1)^{-1} / 2 $. Clearly $ n \ge m $ otherwise there would not be $ d $ possible edges containing a given vertex.} that $ m \le k d^{1/(k - 1)} $ and $ d^{\varepsilon} \le m \le n $ \new{since $ d \le {n-1 \choose k-1} $.}
			Let $ \mathcal{K} $ be the complete $ k $-graph on vertex set $ S $, and define $ \mathcal{H}' $ to be the (disjoint) union of $ \mathcal{H} $ with $ \ell $ vertex-disjoint copies of $ \mathcal{K} $, say $ \mathcal{K}_1, \ldots, \mathcal{K}_{\ell} $ on vertex sets $ S_1, \ldots, S_{\ell} $, respectively.
			Then, by definition, $ \mathcal{H}' $ is essentially regular as required for Theorem~\ref{theorem_cfhm_simple}, and $ \Delta_2(\mathcal{K}) \le m^{k - 2} \le k^{k - 2} d^{1 - 1 / (k - 1)} \le d^{1 - \varepsilon} $, as we can choose $ \varepsilon_0 < (k - 1)^{-1} / 2 $.
			
			For each $ i \in [\ell] $, choose any subset $ T_i \subseteq S_i $ of size $ t \coloneqq 2 d^{\varepsilon / \new{5}} $, and enumerate its vertices arbitrarily as $ v_1^i, \ldots, v_t^i $.
			For each $ j \in [\ell] $, define
			$$ \mathcal{Z}_j \coloneqq \left\{ \{ e_1, \ldots, e_j \}\in { \mathcal{H}' \choose j}: \exists s \in [t] \text{ such that } e_i \cap T_i = \{ v_s^i \} \ \forall i \in [j] \right\} \subseteq { \mathcal{H}' \choose j}. $$
            We claim that the indicator function $ \mathds{1}(\mathcal{Z}_j) $ is a $ (d, \varepsilon / \new{5}, \mathcal{C}) $-trackable \new{$ \ell $-}test function, for each $ j \in [\ell] $.
			
			Clearly every element of $ \mathcal{Z}_j $ is a matching.
			For (W\ref{cond_w1}), observe that $ d^{j + \varepsilon / \new{5}} \le |\mathcal{Z}_j| \le 2 d^{j + \varepsilon / \new{5}} $, since there are $ t=2 d^{\varepsilon / \new{5}} $ possible choices of $ s $, and then at least $ (1 - d^{-\varepsilon / 3}) d $ (and at most $ d $) choices for each edge $ e_i $.
			Indeed, the number of choices for $ e_i $ is exactly the number of edges containing $ v_s^i $ but no other element of $ T_i $; by (H\ref{cond_h_h1codegree}), the number of edges containing $ v_s^i $ and any other element $ u \in T_i $ is bounded by $ 2 d^{\varepsilon / \new{5}} \cdot d^{1 - \varepsilon} \le 2 d^{1 - \varepsilon / \new{5}} $, so the desired bound follows.
			
			For (W$\ref{cond_w2*}^*$), note that given a set $ E' \in {\mathcal{H}' \choose j'} $ for some $ j' \in [j - 1] $, which belongs to some set $ E \in \mathcal{Z}_j $, the choice of $ s \in [t] $ is fixed by $ E' $, as are $ j' $ of the edges of $ E $. Hence $ \mathds{1}(\mathcal{Z}_j)(\{ E \in {\mathcal{H} \choose j} : E \supseteq E' \}) \le d^{j - j'} $.
			The conditions (W\ref{cond_w3}) and (W\ref{cond_w4}) are trivial since there are no conflicts containing edges from $ \mathcal{H}' \setminus \mathcal{H} $.
			
			Now we may define, for each\footnote{Note that the case $ j = 0 $ is trivial.} $ j \in [\ell] $ and each $ j $-uniform \new{semi-}trackable $ w \in \mathcal{Y} $, a new test function $ w'\colon\ {\mathcal{H}\new{'} \choose j} \to [0, 1] $ by $ w'(E) \coloneqq w(E) + \mathds{1}(\mathcal{Z}_j)(E) $ \new{for $ E \subseteq \mathcal{H} $, and $ w'(E) \coloneqq 0 $ otherwise,} and observe that $ w' $ is $ (d, \varepsilon / \new{10}, \mathcal{C}) $-trackable.
			Indeed, $ w'(\mathcal{H}') = w'(\mathcal{H}) \ge |\mathcal{Z}_j| \ge d^{j + \varepsilon / \new{5}} $ so (W\ref{cond_w1}) is satisfied, and (W\ref{cond_w2}) follows from (W$\ref{cond_w2*}^*$) for $ w $ and $ \mathds{1}(\mathcal{Z}_j) $, using the fact that $ \new{(\varepsilon^{-1} + 1)} d^{j - j'} \le d^{-j' - \varepsilon / \new{10}} w'(\mathcal{H}) $.
			The remaining conditions (W\ref{cond_w3}) and (W\ref{cond_w4}) for $ w' $ follow immediately from the fact that they hold for $ w $ and $ \mathds{1}(\mathcal{Z}_j) $.
            \new{Note that $ \mathcal{C} $ is also $ (d, \ell, \varepsilon) $-bounded for $ \mathcal{H}' $ by definition.}
            Since also $ w'(\mathcal{H}) \le 3 d^{j + \varepsilon / \new{5}} $, applying Theorem~\ref{theorem_cfhm_simple} to $ \mathcal{H}' $ with each $ w $ replaced by the corresponding $ w' $ gives us a \new{$ \mathcal{C}$-free matching $ \mathcal{M}' \subseteq \mathcal{H}' $ containing no $(\varepsilon/2, \mathcal{C})$-conflict-sharing pairs} such that $ w'(\mathcal{M}') = (1 \pm d^{-\varepsilon^3}) d^{-j} w'(\mathcal{H}') \le 4 d^{\varepsilon / \new{5}} $.
			
			Let $ \mathcal{M} \coloneqq \mathcal{M}' \cap \mathcal{H} $. Then $ w(\mathcal{M}) \le w'(\mathcal{M}') \le d^{\varepsilon / \new{4}} $ as required.
			Clearly $ \mathcal{M} $ is $ \mathcal{C} $-free and contains no conflict-sharing pairs.
            Since $ |V(\mathcal{H}')| \le (\ell + 1) |V(\mathcal{H})|$, it follows that $ \mathcal{M} $ covers all but a $ d^{-\varepsilon^3} $-fraction of $ \mathcal{H} $.
		\end{proof}
		
	\subsection{The Lovász Local Lemma}
		We now state our other prerequisite, the well-known Lovász Local Lemma, a version of which was originally introduced by Erdős and Lovász \cite{erdos1974problems}.
		The form we use is an immediate corollary to the general form proved by Spencer \cite{spencer1977asymptotic}.

		\begin{lemma}
			\label{lemma_local_mid}
			Let $ \mathfrak{A} = \{ A_1, \ldots, A_n \} $ be a finite set of events in a probability space, and suppose that, for each $ i \in [n] $, there exists a set $ \mathfrak{B}(i) \subseteq [n] $ such that $ A_i $ is mutually independent from $ \{ A_j: j \in [n] \setminus \mathfrak{B}(i) \} $.
			Suppose also that for each $ i \in [n] $ we have $ \mathbb{P}[A_i] < 1 / 2 $ and $ \sum_{j \in \mathfrak{B}(i)} \mathbb{P}[A_j] \le 1 / 4 $.
			Then $ \mathbb{P}[A_1^C \cap \cdots \cap A_n^C] > 0 $.
		\end{lemma}

\subsection{More general conditions} \label{section_general_conditions}
	
	We prove Theorem~\ref{theorem_matching} using slightly weaker conditions  than those given in Section~\ref{section_matching}, allowing the case $ j_2 \new{= |D \cap \mathcal{H}_2|} = 1 $, and giving more freedom to high-degree vertices in $ P $.
    Given a set of edges $ E \subseteq \mathcal{H} $, write
	$$ \gls{edgesetpvertices} \coloneqq \{ y \in P: y \in e \text{ for some } e \in E \cap \mathcal{H}_2 \} $$
	for the set of vertices of $ P $ found in edges of $ \mathcal{H}_2 $ in $ E $.
    \new{A limitation of conditions (D\ref{cond_mixed_conflictsize})--(D\ref{cond_mixed_codegree2}) is that the number of conflicts is governed uniformly by the the minimum degree in $ \mathcal{H}_2 $.
    Intuitively, given some conflict $ E \in \mathcal{D} $, if the vertices in $ V_P(E) $ have much higher degrees in $ \mathcal{H}_2 $, then it should be easier to avoid $ E $ because we have more flexibility when choosing edges of $ \mathcal{H}_2 $ which are used to cover the vertices of $ V_P(E) $ in $ \mathcal{M}_2 $.
    As such, we can afford to have more conflicts involving vertices of $ P $ with higher degrees in $ \mathcal{H}_2 $; this motivates the following quantity, which allows us to weight these conflicts accordingly when bounding their number.}
	
	We define the \textit{unavoidability of $ E $} to be
    \begin{equation} \label{eqn_unavoidability_defn}
        \gls{unavoidability} \coloneqq \prod_{y \in V_P(E)} d_{\mathcal{H}_2}^{-1}(y),
    \end{equation}
	and extend this definition to a conflict hypergraph $ \new{\mathcal{G}} $ by taking the sum over all conflicts, that is
	$ A(\new{\mathcal{G}}) \coloneqq \sum_{E \in \new{\mathcal{G}}} A(E). $
    This has a natural interpretation.
    Suppose we select \new{for each vertex $ v \in P $ one edge in $ \mathcal{H}_2 $ containing $ v $} independently and uniformly at random.
    Then $ A(\new{\mathcal{G}}) $ is the expected number of conflicts in $ \new{\mathcal{G}} $ of which all edges are selected.
    
	Given $ j_1', j_2' \in [\ell] $ and sets $ C \in {\mathcal{H}_1 \choose j_1'} $ and $ D \in {\mathcal{H}_2 \choose j_2'} $, write $ \gls{conflictsextension} \coloneqq \{ E \in \new{\mathcal{G}}: C \cup D \subseteq E \} $.
	Then we may define
    \begin{equation} \label{eqn_unavoidability_degree_defn}
        \gls{maxdegreeunavoidability} \coloneqq \max_{C \in {\mathcal{H}_1 \choose j_1'}, D \in {\mathcal{H}_2 \choose j_2'}} A(\new{\mathcal{G}}_{[C, D]}),
    \end{equation}
	which can be thought of as analogous to the maximum $ (j_1', j_2') $-degree \new{$ \Delta_{j_1', j_2'}(\mathcal{G}) $ (recalling the definition \eqref{eqn:defn_delta_j1j2})}, but with each conflict weighted by its unavoidability.
	Say that \textit{$ \mathcal{D} $ is \gls{mixedbounded}} if the following hold for all $ x, y \in P $ and $ j_1 \in [0, \ell], j_2 \in [\ell] $.
	\begin{enumerate}[(E1)]
		\item $ |D \cap \mathcal{H}_2| \ge 1 $ and $ |D| \in [2, \ell] $ for each conflict $ D \in \mathcal{D} $; \label{cond_mixed_general_conflictsize}
		\item $ A(\mathcal{D}^{(j_1, j_2)}_x) \le d^{j_1 + \delta} $; \label{cond_mixed_general_degree}
		\item $ \Delta^A_{j', 0}(\mathcal{D}^{(j_1, j_2)}_x) \le d^{j_1 - j' - \varepsilon} $ for each $ j' \in [j_1] $; \label{cond_mixed_general_codegree1}
		\item $ A(\mathcal{D}^{(j_1, j_2)}_{x, y}) \le d^{j_1 - \varepsilon} $ whenever $ j_2 \ge 2 $; \label{cond_mixed_general_codegree2}
		\item $ \Delta_{0, 1}(\mathcal{D}^{(j_1, 1)}_x) \le \ell d^{j_1} $; \label{cond_mixed_general_degree_j21}
		\item $ \Delta_{j', 1}(\mathcal{D}^{(j_1, 1)}_x) \le d^{j_1 - j' - \varepsilon} $ for each $ j' \in [j_1 - 1] $. \label{cond_mixed_general_codegree_j21}
	\end{enumerate}

	Observe that the conditions (E\ref{cond_mixed_general_degree}), (E\ref{cond_mixed_general_codegree1}), (E\ref{cond_mixed_general_codegree2}) are analogous to (D\ref{cond_mixed_degree}), (D\ref{cond_mixed_codegree1}), (D\ref{cond_mixed_codegree2}), except conflicts are now weighted by their unavoidability; in particular, for a conflict $ E \in \mathcal{D}^{(j_1, j_2)} $, it is always the case that $ A(E) \le \delta_{P}(\mathcal{H}_2)^{-j_2} $, so $ (d, \ell, \varepsilon) $-simply-bounded implies $ (d, \ell, \varepsilon, \varepsilon^4) $-mixed-bounded.
    \new{
    Indeed, (D\ref{cond_mixed_degree}), (D\ref{cond_mixed_codegree1}), (D\ref{cond_mixed_codegree2}) imply (E\ref{cond_mixed_general_degree}), (E\ref{cond_mixed_general_codegree1}), (E\ref{cond_mixed_general_codegree2}), respectively, using the fact that $A(E) \le |E| \delta_P (\mathcal{H}_2)^{-j_2}$ for every $E \subseteq \mathcal{D}^{(j_1,j_2)}$.
    If (D\ref{cond_mixed_conflictsize}) holds then so does (E\ref{cond_mixed_general_conflictsize}), and in this case both of the conditions (E\ref{cond_mixed_general_degree_j21}) and (E\ref{cond_mixed_general_codegree_j21}) are vacuous.
    }
    Note also that conditions (E\ref{cond_mixed_general_degree_j21}) and (E\ref{cond_mixed_general_codegree_j21}) are stronger versions of (E\ref{cond_mixed_general_degree}) and (E\ref{cond_mixed_general_codegree1}) in the case that $ j_2 = 1 $, requiring small degrees and codegrees for every individual edge $ e \in \mathcal{H}_2 $, rather than just for the sum over all such $ e $ containing a particular vertex $ x \in P $.
	
	In a similar fashion, we may also weaken the degree conditions (H\ref{cond_h_h2degree}) and (H\ref{cond_h_h2codegree}) slightly. Define analogously the \textit{unavoidability of a vertex $ v \in R $} to be
	$$ \gls{unavoidabilityvertex} \coloneqq \sum_{x \in P} \frac{d_{\mathcal{H}_2}(x, v)}{d_{\mathcal{H}_2}(x)} \text{.} $$
	Then we may require only that for all $ x \in P $ and $ v \in R $,
	\begin{enumerate}[(H$1'$)]
		\setcounter{enumi}{2}
		\item $ A(v) \le d^{\varepsilon^4} $ (and in particular $ \delta_{P}(\mathcal{H}_2) \ge 1 $); \label{cond_h_h2d'}
		\item $ d_{\mathcal{H}_2}(x, v) \le d^{-\varepsilon} d_{\mathcal{H}_2}(x) $. \label{cond_h_h2cd'}
	\end{enumerate}
	\new{Recalling that each edge of $ \mathcal{H}_2 $ contains exactly one vertex of $ P $}, note that by definition $A(v) \delta_{P}(\mathcal{H}_2) \le \sum_{x \in P} d(x, v) = d_{\mathcal{H}_2}(v) $ so (H\ref{cond_h_h2degree}) implies (H$\ref{cond_h_h2d'}'$), and clearly (H\ref{cond_h_h2codegree}) implies (H$\ref{cond_h_h2cd'}'$).
	Observe furthermore that (H$\ref{cond_h_h2d'}'$) and (H$\ref{cond_h_h2cd'}'$) together imply that $ \delta_{P}(\mathcal{H}_2) \ge d^{\varepsilon} $; indeed, by (H$\ref{cond_h_h2d'}'$) we must have $ d_{\mathcal{H}_2}(x) \ge 1 $ \new{for every $ x \in P $}, so there is some $ v \in R $ with $ d_{\mathcal{H}_2}(x, v) \ge 1 $, and then by (H$\ref{cond_h_h2cd'}'$) this means $ d_{\mathcal{H}_2}(x) \ge d^{\varepsilon} $.

\section{Proof of Theorem~\ref{theorem_matching}}
	\label{section_matching_proof}

	We assume the setting described in Theorem~\ref{theorem_matching}, and devote the entirety of this section to its proof.
    However, write $ \mathcal{D}' $ in place of $ \mathcal{D} $, and assume that it is a $ (d, \ell, \varepsilon, \varepsilon^4) $-mixed-bounded conflict hypergraph.
    \new{Introduce new constants $ \gls{constants} \in \mathbb{N} $ according to the hierarchy $ 0 < 1/d \ll \varepsilon \ll 1/\ell' \ll 1/\Gamma \ll 1/i^* \ll 1/\ell, 1/k $.}

    \subsection{\new{Proof outline}}

    Recall our two-stage proof method outlined in Section~\ref{section_matching}.
    We first apply Theorem~\ref{theorem_cfhm_mod} to $ \mathcal{H}_1 $ to give a conflict-free matching $ \mathcal{M}_1 \subseteq \mathcal{H}_1 $ covering most of the vertices of $ P $.
	We then extend this to a covering of $ P $ by randomly choosing, for each vertex $ x \in P $ which is not already covered by $ \mathcal{M}_1 $, some edge from $ \mathcal{H}_2 $ to cover $ x $; call the resulting set of edges $ \mathcal{M}_2 \subseteq \mathcal{H}_2 $.
    Using the Lovász Local Lemma, we show that with non-zero probability $ \mathcal{M}_2 $ is indeed a matching, and the full matching $ \mathcal{M} \coloneqq \mathcal{M}_1 \cup \mathcal{M}_2 $ is conflict-free.
    \new{This approach presents several challenges which must be overcome.}

    \new{Firstly, Theorem~\ref{theorem_cfhm_mod} requires $ \mathcal{H}_1 $ to be essentially vertex-regular, but (H\ref{cond_h_h1degree}) provides no lower bound on the degrees of vertices in $ Q $.}
    We solve this \new{in Section~\ref{section_proof_regularising}} by adding dummy vertices and edges to increase the vertex degrees in $ Q $.
    \new{Secondly, it need not be the case that the randomly chosen edges of $ \mathcal{M}_2 $ form a matching.
    We solve this in Section~\ref{section_proof_matching}} by simply adding non-disjoint pairs of edges in $ \mathcal{H}_2 $ as additional conflicts of size 2, replacing $ \mathcal{D}' $ by a larger conflict hypergraph $ \mathcal{D} $, \new{and thus forcing $ \mathcal{M}_2 $ to be a matching.}

    \new{The next challenge is to ensure that we can choose $ \mathcal{M}_2 $ such that $ \mathcal{M} $ is conflict-free.
    The conflicts we must avoid are of the form $ C \cup D $, where two conditions hold:} $ C \subseteq \mathcal{M}_1 $ and no edge in $ D $ contains a vertex of $ P $ which has already been chosen in $ \mathcal{M}_1 $.
    \new{If the second condition does not hold, we say that $ D $ is \textit{blocked}.
    In order to ensure that no such $ D $ is contained in $ \mathcal{M}_2 $, we seek to bound from above the total number of unblocked potential conflicts involving a given vertex $ x \in P $ in some edge of $ D \subseteq \mathcal{H}_2 $;} since most vertices of $ P $ are already covered by $ \mathcal{M}_1 $, this significantly reduces the number of conflicts we need to consider.
    In Section~\ref{section_proof_conflicts}, we define test functions to track the number of partial conflicts appearing in the matching $ \mathcal{M}_1 $, as well as the number of those which are blocked; by taking the difference of these two values we later obtain the desired bound.

    \new{The final challenge to overcome} is that this approach to blocking conflicts fails when $ j_2 = 1 $, since the conflicts contain no other edges of $ \mathcal{H}_2 $ which might be blocked.
    \new{We therefore require the stricter bound $ \ell d^{j_1} $ in (E\ref{cond_mixed_general_degree_j21}) on the number of such} conflicts of a given size $ j_1 $ containing a given edge $ e \in \mathcal{H}_2 $.
    Under the assumption that edges in $ \mathcal{H}_1 $ appear in the matching $ \mathcal{M}_1 $ `pseudorandomly with probability $ d^{-1} $' in some suitable sense, this allows the `expected number' of potential conflicts (of any size) for each $ e $ to be up to $ \ell^2 $; so choosing $ e $ purely at random still will not suffice.
    \new{To overcome this, we observe that,} assuming that the Poisson paradigm applies to our concept of pseudorandomness, there should be a constant proportion (specifically, at least $ \exp(-\ell^2) $) of the edges $ e $ containing any given vertex $ x $ with no potential conflicts arising from $ \mathcal{M}_1 $.
    Hence for each vertex $ x \in P $ we may restrict our random choice of $ \mathcal{M}_2 $ to only these \textit{safe} edges.
    We encode this `pseudorandom' behaviour in a further set of test functions, defined in Section~\ref{section_matching_proof_j21}.
    
    \new{In Section~\ref{section_proof_construction}}, we apply Theorem~\ref{theorem_cfhm_mod} to obtain our matching $ \mathcal{M}_1 $, randomly choose $ \mathcal{M}_2 $, and use the previously defined test functions to bound the number of potential conflicts in each case, so that $ \mathcal{M}_2 $ is conflict-free with non-zero probability.
	
	\subsection{Regularising $ \mathcal{H}_1 $}
	\label{section_proof_regularising}
	
		First, in order to apply Theorem~\ref{theorem_cfhm_mod} to the hypergraph $ \mathcal{H}_1 $, we need to ensure that all vertices in $ \mathcal{H}_1 $ have degree roughly $ d $, rather than just those in $ P $; since we only require a $ P $-perfect matching, and make no statement about which vertices of $ Q $ are covered, we can achieve this by simply adding dummy edges to boost the degrees of vertices in $ Q $.
		
		Assume that $ q > 0 $ (otherwise we may skip this step), and let $ m \coloneqq |Q| $.
        For each vertex $ v \in Q $, let $ d' \coloneqq d - d_{\mathcal{H}_1}(v) $ and add $ d' $ new edges, each containing $ v $ and a set of $ k - 1 $ new vertices, such that each new vertex is only ever contained in one edge.
        Refer to the set of new vertices as $ Q' $ and the set of new edges as $ \mathcal{E}_Q $.
        Add further new vertices until $ Q' $ contains $ |Q| dm $ vertices.
        Next, let $ \mathcal{F}' $ be a binomial random $ k $-graph with vertex set $ Q' $ and expected degree $ (1 - d^{-\varepsilon} / 2) d $.
        It is routine to show that, \new{with high probability,} $ (1 - d^{-\varepsilon})d \le d_{\mathcal{F}'}(v) \le d - 1 $ and $ d_{\mathcal{F}'}(u, v) \le d^{1 - \varepsilon} $ for all distinct $ u, v \in Q' $\new{; take $ \mathcal{F} $ to be a graph satisfying these properties.}
        Note that each $ v \in Q' $ belongs to at most one edge in $ \mathcal{E}_Q $, so in total $ d_{\mathcal{H}}(v) \le d $.
        Define the hypergraph $ \mathcal{H}_1' \coloneqq \mathcal{H}_1 \cup \mathcal{E}_Q \cup \mathcal{F} $, which we will use in place of $ \mathcal{H}_1 $ later in the proof.

	\subsection{Ensuring a Matching in $ \mathcal{H}_2 $}
	\label{section_proof_matching}

        Let
        $ \mathcal{E} \coloneqq \{ \{ e, f \} \subseteq \mathcal{H}_2 : \emptyset \ne e \cap f \subseteq R \} $
        be the conflict hypergraph containing as conflicts all pairs of edges in $ \mathcal{H}_2 $ that overlap (only) in $ R $.
        We show that $\mathcal{E}$ is a $ (d, \ell, \varepsilon \new{/2}, \new{2}\varepsilon^4) $-mixed bounded hypergraph by verifying that it satisfies (E\ref{cond_mixed_general_conflictsize})-(E\ref{cond_mixed_general_codegree_j21}).
        Note that in particular this implies that $ \mathcal{D} \coloneqq \mathcal{D}' \cup \mathcal{E} $ is $ (d, \ell, \varepsilon \new{/3}, \new{3} \varepsilon^4) $-mixed-bounded\new{; indeed, (E\ref{cond_mixed_general_conflictsize}) is trivial, the bounds in conditions (E\ref{cond_mixed_general_degree}), (E\ref{cond_mixed_general_codegree1}), (E\ref{cond_mixed_general_codegree2}) are additive and the bounds in (E\ref{cond_mixed_general_degree_j21}) and (E\ref{cond_mixed_general_codegree_j21}) are unaffected by $ \mathcal{E} $.}
        
		Indeed, the condition (E\ref{cond_mixed_general_conflictsize}) is clear and (E\ref{cond_mixed_general_codegree1}), (E\ref{cond_mixed_general_degree_j21}), (E\ref{cond_mixed_general_codegree_j21}) are trivial.
		For (E\ref{cond_mixed_general_degree}), note that \new{we need only consider the case} $(j_1, j_2)=(0,2)$ and suppose that $ D = \{ e, f \} $ is a conflict containing $ x \in P $; say $ x \in e $, and let $ v \in e \cap f \cap R $ and $ y \in f \cap P $, noting that $x \ne y $.
		The number of such conflicts $ D $, given $ x, v, y $, is exactly $ d_{\mathcal{H}_2}(x, v) d_{\mathcal{H}_2}(y, v) $. Moreover, \new{recalling definitions \eqref{eqn_unavoidability_defn} and \eqref{eqn_unavoidability_degree_defn}}, as $|e \cap P|=|f \cap P|=1$, we have $ A(D) = d_{\mathcal{H}_2}(x)^{-1} d_{\mathcal{H}_2}(y)^{-1} $, so we obtain
		$$ A(\mathcal{E}_x) = \sum_{D \in \mathcal{E}_x} A(D) \le \sum_{v \in R} \sum_{y \in P} \frac{d_{\mathcal{H}_2}(x, v) d_{\mathcal{H}_2}(y, v)}{d_{\mathcal{H}_2}(x) d_{\mathcal{H}_2}(y)}
			= \sum_{v \in R} \frac{d_{\mathcal{H}_2}(x, v)}{d_{\mathcal{H}_2}(x)} A(v) \le \new{r} \cdot d^{\varepsilon^4}, $$
        which suffices, using the fact that $ A(v) \le d^{\varepsilon^4} $ by (H$\ref{cond_h_h2d'}'$), \new{and noting that we allow the case $ y = x $ in the sum, as we seek only an upper bound.}
		For (E\ref{cond_mixed_general_codegree2}), observe that similarly
		$$ A(\mathcal{E}_{x, y}) \le \sum_{v \in R} \frac{d_{\mathcal{H}_2}(x, v)}{d_{\mathcal{H}_2}(x)} \frac{d_{\mathcal{H}_2}(y, v)}{d_{\mathcal{H}_2}(y)}
			\stackrel{(\text{H}\ref{cond_h_h2cd'}')}{\le} d^{-\varepsilon} \sum_{v \in R} \frac{d_{\mathcal{H}_2}(x, v)}{d_{\mathcal{H}_2}(x)}
			= \new{r} d^{-\varepsilon}, $$
		which suffices.
		The rest of the proof will show that all of the conflicts in $ \mathcal{D} $ can indeed be avoided when choosing $ \mathcal{M}_2 $.

	\subsection{Tracking mixed conflicts: the case $ j_2 \ge 2 $}
	\label{section_proof_conflicts}
        In this section, we define a set of test functions and show that they are either trackable or \new{semi-}trackable.
        These will be useful in Section~\ref{section_bound_unblocked_conflicts} for showing that the expected number of unblocked conflicts can be appropriately bounded from above such that the Lovász Local Lemma can be used to avoid them.
		
		For the remainder of this section, we fix $ j_1 \in [0, \ell] $ and $ j_2 \in [2, \ell] $, and consider only conflicts with $ j_1 $ and $ j_2 $ edges from $ \mathcal{H}_1 $ and $ \mathcal{H}_2 $ respectively; for ease of notation, let $ \mathcal{G} = \mathcal{D}^{(j_1, j_2)} $ \new{(recalling the notation from Section~\ref{section:simply_bounded})}.
		For $i \in [2]$, we write
		$ \gls{hiparts} \coloneqq \{ E \cap \mathcal{H}_i : E \in \mathcal{G} \} $
		for the set of \textit{$ \mathcal{H}_i $-parts} of conflicts in $ \mathcal{G} $.

	\subsubsection{Defining test functions} \label{section_deftestfn}

        We extend our previous notation \new{from Section~\ref{section_general_conditions}} by writing $ \gls{conflictscontc} = \mathcal{G}_{[C, \emptyset]} $, $ \gls{conflictscontcx} = (\mathcal{G}_{[C]})_{x} $ and $ \gls{conflictscontcxy} = (\mathcal{G}_{[C]})_{x, y} $, for any $ C \subseteq \mathcal{H}_1 $ and $ x, y \in P $.
        For the remainder of this section, fix a vertex $ x \in P $.
		We first define $ w_x\colon\ {\mathcal{H}_1 \choose j_1} \to \mathbb{R}_{\ge 0} $ by
		$$ \gls{testfnwx}(C) \coloneqq A(\mathcal{G}_{[C], x}), $$
        \new{recalling \eqref{eqn_unavoidability_defn}.}
        Given a matching $ \mathcal{M}_1 \subseteq \mathcal{H}_1 $ containing $ C $ (but not covering $ x $), if we choose $ \mathcal{M}_2 \subseteq \mathcal{H}_2 $ randomly, then the function $ w_x(C) $ represents the expected number of conflicts from $ \mathcal{G} $ present in $ \mathcal{M} = \mathcal{M}_1 \cup \mathcal{M}_2 $ which have $ C $ as their $ \mathcal{H}_1 $-part and contain $ x $ in their $ \mathcal{H}_2 $-part.
        As such, $ w_x(\mathcal{M}_1) $ represents the expected number of conflicts containing $ x $ in their $ \mathcal{H}_2 $-part for which the $ \mathcal{H}_1 $-part is \new{covered by} $ \mathcal{M}_1 $.
        
        Next, in order to ensure that the test functions we define are trackable, we make an intermediate definition.
        Say that a set of edges $ E \subseteq \mathcal{H} $ is \textit{\gls{testable}} if it satisfies the following conditions:
        \begin{itemize}
            \item $ E \cap \mathcal{H}_1 $ is a matching;
            \item $ E $ is $ \mathcal{C} $-free;
            \item $ E $ contains no $ \varepsilon / 2 $-conflict-sharing pairs of edges (as defined in Section~\ref{section_proof_background_cfhm}).
        \end{itemize}
        Say that $ E $ is \textit{untestable} if it is not testable.\footnote{\new{This name reflects the fact that we must ignore such sets when defining our test functions, in order to ensure that the functions are trackable (as per the definition in Section~\ref{section_proof_background_cfhm}).}}
        We may \new{(and do henceforth)} assume without loss of generality that all of the conflicts in $ \mathcal{D} $ are testable.
        \new{Indeed, letting $ \overline{\mathcal{D}} \subseteq \mathcal{D} $ be the set of all testable conflicts, it is clear that $ \overline{\mathcal{D}} $ is also $ (d, \ell, \varepsilon/3, 3\varepsilon^4) $-mixed-bounded.
        Furthermore, if we are able to find a matching $ \mathcal{M}_2 \subseteq \mathcal{H}_2 $ such that $ \mathcal{M} = \mathcal{M}_1 \cup \mathcal{M}_2 $ is $ \overline{\mathcal{D}} $-free, with $ \mathcal{M}_1 \subseteq \mathcal{H}_1 $ obtained from Theorem~\ref{theorem_cfhm_mod}, then $ \mathcal{M}_1 $ may not contain any untestable set, which in particular means that $ \mathcal{M} $ is also $ \mathcal{D} $-free.}
        Similarly we may assume that $ w_x(C) = 0 $ for any $ C $ containing $ x $\new{, because if $ C $ is covered by $ \mathcal{M}_1 $ then no edge containing $ x $ can be chosen in $ \mathcal{M}_2 $.}
		
        We may now define further $ w_x'\colon\ {\mathcal{H}_1 \choose j_1 + 1} \to \mathbb{R}_{\ge 0} $ by
		$$ \gls{testfnwx'}(C') \coloneqq \mathds{1}(C' \text{ testable}) \sum_{e \in C'} \sum_{y \in (e \cap P) \setminus \{ x \}} A(\mathcal{G}_{[C' \setminus \{ e \}], x, y}). $$
		\new{Observe that} the function $ w_x'(C') $ represents the sum, over each possible partition $ C' = C \cup \{ e \} $, of the expected number of those conflicts counted by $ w_x(C) $ whose $ \mathcal{H}_2 $-part shares some vertex $ y \in P \setminus \{ x \} $ with the edge $ e \in \mathcal{H}_1 $.
        In particular this means $ w_x'(\mathcal{M}_1) $ represents the expected number of those conflicts counted by $ w_x(\mathcal{M}_1) $ whose $ \mathcal{H}_2 $-part shares some vertex $ y \in P $ with some edge $ e \in \mathcal{M}_1 $.
        \new{In our proof, we will consider the quantity $ w_x(\mathcal{M}_1) - w_x'(\mathcal{M}_1) $, which counts the number of conflicts (containing $ x $ in their $ \mathcal{H}_2 $-part) for which the $ \mathcal{H}_1 $-part is contained in $ \mathcal{M}_1 $, and the $ \mathcal{H}_2 $-part is at risk of being contained in $ \mathcal{M}_2 $.
        By bounding this quantity from above, we can ensure that all such conflicts can be avoided when choosing $ \mathcal{M}_2 $.}
		Observe also that both function definitions still make sense when $ j_1 = 0 $, in which case $ w_x $ is defined only on the empty set, and $ w_x' $ is defined on single edges.

        \subsubsection{Ignoring untestable sets}

        The goal of this section is to prove (\ref{eqn_bad_blockers_bounds}) below.
		For any \new{given} $ C \in E_1(\mathcal{G}) $, and any vertex $ y \in P $ not contained in an edge of $ C $, we write
		$ \gls{blockingedges} \coloneqq |\{e \in \mathcal{H}_1 : y \in e \text{ and } C \cup e \text{ is testable} \}|, $
		and claim that
		\begin{equation} \label{eqn_bad_blockers_bounds}
			(1 - d^{-\varepsilon / 3}) d \le b_{C, y} \le d.
		\end{equation}
  
		In other words, restricting to testable sets only removes a negligible proportion of the total number of edges $ e $ containing $ y $ for which $ w_x'(C \cup \{ e \}) $ \new{would otherwise be positive}.
        Indeed, we estimate $ b_{C, y} $ by bounding \new{from} above the number of $ e \in \mathcal{H}_1 $ such that $ y \in e $ but $ C \cup e $ is untestable.
		\new{\textbf{Firstly}}, for an $ \mathcal{H}_1 $-part $ C $ of size $ j_1 $, observe that $ C $ covers in total $ k j_1 $ vertices $ z $, and for each of these there are at most $ \Delta_2(\mathcal{H}_1) $ edges $ e \in \mathcal{H}_1 $ with $ y, z \in e $.
        Thus the total number of $ e \in \mathcal{H}_1 $ with $ y \in e $ for which $ C \cup e $ is not a matching is at most $ j_1 k d^{1 - \varepsilon} $ by (H\ref{cond_h_h1codegree}).
		\new{\textbf{Secondly}}, by (C\ref{cond_c3}), the number of such $ e $ for which $ C \cup e $ contains a conflict from $ \mathcal{C} $ is at most
		$$ \sum_{j = 2}^{\ell} \sum_{F \in {C \choose j - 1}} \Delta_{j - 1}(\mathcal{C}^{(j)}) \le \ell 2^{j_1} d^{1 - \varepsilon}, $$
        \new{where $ F $ represents a subset of $ C $ for which $ F \cup e $ is a conflict of size $ j $.\footnote{Recall here that $ C $ is testable by our assumption on $ \mathcal{D} $, so this is the only way in which $ C \cup \{ e\} $ may contain a conflict.}}
		\new{\textbf{Thirdly}}, \new{we bound the number of $ e $ for which $ C \cup e $ contains a conflict-sharing pair.\footnote{Note that $ e $ must be one element of this pair, by testability of $ C $.}}
        For each edge $ f \in \mathcal{H}_1 $ and $ j \in [\ell] $, let
		$ P_f^j \coloneqq \{ e \in \mathcal{H}_1: |(\mathcal{C}_e)^{(j)} \cap (\mathcal{C}_f)^{(j)}| \ge d^{j - \varepsilon / 2} \}. $
        Observe that by definition we may rewrite
        $$ \sum_{e \in P_f^j} |(\mathcal{C}_e)^{(j)} \cap (\mathcal{C}_f)^{(j)}| =
            \sum_{C' \in (\mathcal{C}_f)^{(j)}} \sum_{e \in P_f^j} \mathds{1}(C' \cup e \in \mathcal{C}^{(j + 1)}) \le
            \sum_{C' \in (\mathcal{C}_f)^{(j)}} |\{ e \in \mathcal{H}_1 : C' \cup e \in \mathcal{C}^{(j + 1)} \}|, $$
        and applying conditions (C\ref{cond_c2}) and (C\ref{cond_c3}), we obtain that
        $$ \sum_{C' \in (\mathcal{C}_f)^{(j)}} |\{ e \in \mathcal{H}_1 : C' \cup e \in \mathcal{C}^{(j + 1)} \}| \le
        \Delta(\mathcal{C}^{(j)}) \Delta_{j}(\mathcal{C}^{(j + 1)}) \le
        d^{j + 1 - \varepsilon}, $$
        \new{where the second expression is the product of the number of choices for $ C' $ given $ f $ and the number of choices for $ e $ given $ C' $.}
		Then, by the definition of $ P_f^j $, we see that
		\begin{equation} \label{eqn_bound_badpairs}
			|P_f^j| \le
			d^{\varepsilon / 2 - j} \sum_{e \in P_f^j} |(\mathcal{C}_e)^{(j)} \cap (\mathcal{C}_f)^{(j)}| \le
			d^{1 - \varepsilon / 2}.
		\end{equation} 
        Hence, summing over each $ f \in C $ and $ j \in [\ell] $, this means that in total there are at most $ j_1 \ell d^{1 - \varepsilon / 2} $ edges $ e \in \mathcal{H}_1 $ for which $ C \cup \{ e \} $ contains a conflict-sharing pair.
        We finish by subtracting the three bounds we have obtained from (H\ref{cond_h_h1degree}) to see that (\ref{eqn_bad_blockers_bounds}) holds.
		
		\subsubsection{Estimating values of test functions}
        \label{section_estimating_tf_values}

        \new{In this section we use \eqref{eqn_bad_blockers_bounds} to show that ignoring untestable sets in the definition of $ w_x' $ does not significantly affect the value of $ w_x'(\mathcal{H}_1) $, compared to what its value would be if they were included.}
        Specifically, we evaluate $ w_x(\mathcal{H}_1) $ and $ w_x'(\mathcal{H}_1) $ and show that $ w_x'(\mathcal{H}_1) \approx (j_2 - 1) d w_x(\mathcal{H}_1) $.
        \new{Recalling the definition of $ V_P(\cdot) $ from Section~\ref{section_general_conditions}}, begin by rewriting\COMMENT{Note that summing over all $ e \in \mathcal{H}_1 $ does not overcount, because if $ e \in C $ and $ y \in e $ then $ \mathcal{G}_{[C], x, y} $ is empty.}
		\begin{align*}
			w_x'(\mathcal{H}_1) &= \sum_{C \in {\mathcal{H}_1 \choose j_1}} \sum_{e \in \mathcal{H}_1 \new{\setminus C}} \mathds{1}(C \cup e \text{ testable}) \sum_{y \in (e \cap P) \setminus \{ x \}} A(\mathcal{G}_{[C], x, y})\\
                &= \sum_{C \in {\mathcal{H}_1 \choose j_1}} \sum_{e \in \mathcal{H}_1 \new{\setminus C}} \mathds{1}(C \cup e \text{ testable}) \sum_{y \in (e \cap P) \setminus \{ x \}} \sum_{D \in E_2(\mathcal{G}_{[C], x})} \mathds{1}(y \in V_P(D)) A(D) \\
				&= \sum_{C \in {\mathcal{H}_1 \choose j_1}} \sum_{D \in E_2(\mathcal{G}_{[C], x})} A(D) \sum_{y \in V_P(D) \setminus \{ x \}} |\{e \in \mathcal{H}_1 \new{\setminus C} : y \in e \text{ and } C \cup e \text{ testable} \}|\\
                &= \sum_{E \in \mathcal{G}_x} A(E) \sum_{y \in V_P(E) \setminus \{ x \}} b_{\new{E \cap \mathcal{H}_1},y}.
		\end{align*}
        \new{Note here that $ E = C \cup D $ and recall that $ \mathcal{G} \coloneqq \mathcal{D}^{(j_1, j_2)} $, so there is an obvious bijection between the set of pairs $ C, D $ in the sum and the set $ \mathcal{G}_x $.}
        Now observe that by definition
		$ w_x(\mathcal{H}_1) = A(\mathcal{G}_x) = \sum_{E \in \mathcal{G}_x} A(E) $.\COMMENT{$ w_x(\mathcal{H}_1) = \sum_{C \in {\mathcal{H}_1 \choose j_1}} A(\mathcal{G}_{[C], x}) = A(\mathcal{G}_x) $ since $ \mathcal{G}_x $ is the disjoint union of the $ \mathcal{G}_{[C], x} $.}
		Therefore, recalling that $ |V_P(E) \setminus \{ x \}| = j_2 - 1 $ and applying each of the two bounds in (\ref{eqn_bad_blockers_bounds}) to the inner sum above, we obtain that
		\begin{equation} \label{eqn_wx'_bounds}
			(1 - d^{-\varepsilon / 3}) (j_2 - 1) d w_x(\mathcal{H}_1) \le w_x'(\mathcal{H}_1) \le (j_2 - 1) d w_x(\mathcal{H}_1).
		\end{equation}
		
		We use these estimates in Section~\ref{section_bound_unblocked_conflicts} to ensure that the choice of $ \mathcal{M}_1 $ leaves very few unblocked potential conflicts, by showing that $ (j_2 - 1) w_x(\mathcal{M}_1) \approx w_x'(\mathcal{M}_1) $, so that most potential conflicts are in fact blocked; this will give an upper bound on the expected number of unblocked conflicts containing $ x $ which are present in $ \mathcal{M}_1 $.

		\subsubsection{Checking test function conditions}

        In order to use Theorem~\ref{theorem_cfhm_mod} to track the values of $ w_x $ and $ w_x' $, we must first scale them appropriately so that they take values in \new{the interval $ [0, \ell'] $.}
        In order to do this, define
		$ \gls{alphax} \coloneqq \max \{ \alpha_x', \alpha_x'' \}, $
		where
		$$ \alpha_x' \coloneqq \max_{j' \in [j_1]} d^{j' - j_1} \Delta^A_{j', 0}(\mathcal{G}_x),
		\text{ and }
		\alpha_x'' \coloneqq d^{-j_1} \max_{y \in P} A(\mathcal{G}_{x, y}), $$
        recalling the definitions in (\ref{eqn_unavoidability_defn}) and (\ref{eqn_unavoidability_degree_defn}).
		Note that
		$ \alpha_x \le d^{-\varepsilon \new{/3}} $
		by (E\ref{cond_mixed_codegree1}) and (E\ref{cond_mixed_codegree2}).
        We show now that either both of the functions $ \alpha_x^{-1} w_x $ and $ \alpha_x^{-1} w_x' $ are $ (d, \varepsilon / \new{10}, \mathcal{C}) $-trackable, or the function $ \alpha_x^{-1} w_x $ is $ (d, \varepsilon / \new{10}, \mathcal{C}) $-\new{semi-}trackable.
        \new{Observe that we need not consider the case $ \alpha_x = 0 $, since in this case $ w_x = w_x' = 0 $ are trivially trackable.}

        Firstly, the case $ j' = j_1 $ in $ \alpha_x' $ ensures that $ \alpha_x^{-1} w_x(C) \le 1 \new{\le \ell'} $ for all $ C \in {\mathcal{H}_1 \choose j_1} $.\COMMENT{For $ |C| = j_1 $ have $ A(\mathcal{G}_{[C], x}) \le \Delta^A_{j_1, 0}(\mathcal{G}_x) $ by definition of $ \Delta^A $, and $ \alpha_x $ is at least this by definition, so $ \alpha_x^{-1} w_x(C) \le \alpha_x^{-1} \alpha_x = 1 $.}
        Also $ \alpha_x^{-1} w_x'(C') \le \alpha_x^{-1} (j_1 + 1) p \Delta^A_{j_1, 0}(\mathcal{G}_x) \le (j_1 + 1) p \new{\le \ell'} $ by the definition of $ \alpha_x' $, recalling that $ \ell' \ll 1/\ell, 1/k $.
        Clearly both $ \alpha_x^{-1} w_x $ and $ \alpha_x^{-1} w_x' $ are zero on any set of edges which is not a matching\new{, recalling the assumption that all conflicts in $ \mathcal{G} $ are testable.}, so they are indeed both \new{$ \ell' $-}test functions \new{(which are $ j_1 $-uniform and $ (j_1 + 1) $-uniform, respectively)}.

        Observe that conditions (W\ref{cond_w3}) and (W\ref{cond_w4}) for $ \alpha_x^{-1} w_x $ are immediate from the assumption that no element of $ \mathcal{G} $ contains a conflict from $ \mathcal{C} $ or conflict-sharing pair, and they both hold for $ \alpha_x^{-1} w_x' $ by definition.
        We prove that both functions satisfy (W$\ref{cond_w2}^*$).
        Given $ j' \in [0, \ell] $ and $ E \in {\mathcal{H}_1 \choose j'} $, write
		$ \gls{edgesetextensions} \coloneqq \{ Z \in {\mathcal{H}_1 \choose j}: Z \supset E \} $
        for each $ j \in [j', \ell] $.
		Then for $ j' \in [j_1 - 1] $, we have $ \alpha_x^{-1} w_x(\mathcal{Z}_E^{(j_1)}) = \alpha_x^{-1} A(\mathcal{G}_{[E], x}) \le d^{j_1 - j'} $ by the definition of $ \alpha_x' $, so (W$\ref{cond_w2}^*$) holds for $ w_x $.

        For $ \alpha_x^{-1} w_x'$, let $ j' \in [j_1 + 1 - 1] = [j_1] $, consider a sum over $ C' = C \cup \{ e \} $ with $ E \subseteq C' $, and split into two cases depending on whether $ E \subseteq C $. Specifically, write
		\begin{equation} \label{eqn_wx'_w2split}
			w_x'(\mathcal{Z}_E^{(j_1 + 1)}) \le \sum_{C' \in \mathcal{Z}_E^{(j_1 + 1)}} \sum_{e \in C'} \sum_{y \in (e \cap P) \setminus \{ x \}} A(\mathcal{G}_{[C' \setminus \{ e \}], x, y}) \le S_1 + S_2,
		\end{equation}
		where \COMMENT{Note these sums include the case $ C' $ not a matching, or $ e \in C $; but as an upper bound they suffice.}
		$$ S_1 \coloneqq \sum_{C \in \mathcal{Z}_E^{(j_1)}} \sum_{e \in \mathcal{H}_1} \sum_{y \in (e \cap P) \setminus \{ x \}} A(\mathcal{G}_{[C], x, y}), $$
		and
		$$ S_2 \coloneqq \sum_{e \in E} \sum_{C \in \mathcal{Z}_{E \setminus \{ e \}}^{(j_1)}} \sum_{y \in (e \cap P) \setminus \{ x \}} A(\mathcal{G}_{[C], x, y}). $$
		
		Firstly for the case $ E \subseteq C $, we may rearrange sums to see that
        \begin{align*}
            S_1 &= \sum_{C \in \mathcal{Z}_E^{(j_1)}} \sum_{e \in \mathcal{H}_1} \sum_{y \in P \setminus \{ x \}} \sum_{D \in E_2(\mathcal{G}_{[C], x})} \mathds{1}(y \in V_P(D), y \in e) A(D)\\
                &= \sum_{C \in \mathcal{Z}_E^{(j_1)}} \sum_{D \in E_2(\mathcal{G}_{[C], x})} A(D)  \sum_{y \in V_P(D) \setminus \{ x \}} d_{\mathcal{H}_1}(y).
        \end{align*}
        We can bound this \new{from} above by
		\begin{equation} \label{eqn_wx'_w2s1}
		  S_1 \le (j_2 - 1) d \sum_{C \in \mathcal{Z}_E^{(j_1)}} \sum_{D \in E_2(\mathcal{G}_{[C], x})} A(D) = (j_2 - 1) d w_x(\mathcal{Z}_E^{(j_1)}) \le (j_2 - 1) \alpha_x d^{j_1 - j' + 1},
		\end{equation}
        using the fact that $ d_{\mathcal{H}_1}(y) \le d $ and $ |V_P(D) \setminus \{ x \}| = j_2 - 1 $, and then the definition of $ w_x $ and the condition (W$\ref{cond_w2}^*$) for $ w_x $.
        
		For the case that $ e \in E $, we in fact split into two subcases\new{, according to whether $ j' > 1 $.}
		If $ j' > 1 $, then we can rearrange and use (W$\ref{cond_w2}^*$) for $ w_x $ similarly to see that
		\begin{equation} \label{eqn_wx'_w2s2c1}
			S_2 \le \sum_{e \in E} \sum_{C \in \mathcal{Z}_{E \setminus \{ e \}}^{(j_1)}} p A(\mathcal{G}_{[C], x}) = p \sum_{e \in E} w_x(\mathcal{Z}_{E \setminus \{ e \}}^{(j_1)}) \le j' p \alpha_x d^{j_1 - j' + 1}.
		\end{equation}
		If however $ j' = 1 $, then $ E = \{ e \} $ and we are summing over all possible $ C \in {H_1 \choose j_1} $, which is why we need the definition of $ \alpha_x'' $ to give
		\begin{equation} \label{eqn_wx'_w2s2c2}
			S_2 = \sum_{y \in (e \cap P) \setminus \{ x \}} A(\mathcal{G}_{x, y}) \le p \alpha_x d^{j_1 \new{-j' + 1}}.
		\end{equation}
		\begin{sloppypar}
			\new{Plugging} the bounds (\ref{eqn_wx'_w2s1}), (\ref{eqn_wx'_w2s2c1}), and (\ref{eqn_wx'_w2s2c2}) back into (\ref{eqn_wx'_w2split}) gives
		  $ \alpha_x^{-1} w_x'(\mathcal{Z}_E^{(j_1 + 1)}) \le C_{\ell, k} d^{j_1 + 1 - j'} $
		  for some constant $ C_{\ell, k} $ depending only on $ \ell $ and $ k $, as required for (W$\ref{cond_w2}^*$).
		\end{sloppypar}
        
        Hence if $ \alpha_x^{-1} w_x(\mathcal{H}_1) \le d^{j_1 + \varepsilon / \new{5}} $, then $ \alpha_x^{-1} w_x $ is $ (d, \varepsilon / \new{10}, \mathcal{C}) $-\new{semi-}trackable.
        If instead $ \alpha_x^{-1} w_x(\mathcal{H}_1) \ge d^{j_1 + \varepsilon / \new{5}} $, then by the estimate (\ref{eqn_wx'_bounds}) we have $ \alpha_x^{-1} w_x'(\mathcal{H}_1) \ge d^{j_1 + 1 + \varepsilon / \new{10}} $ so (W\ref{cond_w1}) holds for both functions.
        In this case, the condition (W$\ref{cond_w2*}^*$) implies (W\ref{cond_w2}), so they are both $ (d, \varepsilon / \new{10}, \mathcal{C}) $-trackable.
	
	\subsection{Tracking mixed conflicts: the case $ j_2 = 1 $}
	\label{section_matching_proof_j21}

		In this section, we suppose $ j_2 = 1 $, noting that $ j_1 \ge 1 $ by (E\ref{cond_mixed_general_conflictsize}), and define a further set of test functions; these are used in Section~\ref{section_restrict_safe} to show that such conflicts can be avoided.
        We fix $ x \in P $ for the remainder of this section, and for ease of notation we write $ \gls{nhoodxh2} \coloneqq \{ e \in \mathcal{H}_2 : x \in e \} $ and $ \gls{degreexh2} \coloneqq d_{\mathcal{H}_2}(x) = |N_x| $.
        Furthermore, in this section, we define $ \mathcal{G} \coloneqq \cup_{j_1 \in [\ell]} \mathcal{D}^{(j_1, 1)} $, \new{noting that we now consider only $ j_2 = 1 $, but for all values of $ j_1 $ together.}
				
		\subsubsection{Defining test functions} \label{section_j21_testdef}

        \new{Recall that we fixed a constant $ i^* \in \mathbb{N} $ with $ 1/i^* \ll 1/\ell $.
        Let $ m \in [i^*] $, $ s' \in [m \ell] $, $ s \in [s'] $, and consider any collection $ \overline{C} \coloneqq \{ C_t \}_{t \in [m]} $ of $ m $ edge sets $ C_t \subseteq \mathcal{H}_1 $, each with size $ |C_t| \in [\ell] $, such that $ \sum_{t \in [m]} |C_t| = s' $ but $ |\bigcup_{t \in [m]} C_t | = s $.
        Let $ \gls{intersectionsencoding} \coloneqq (|\bigcap_{t \in T} C_t|)_{T \subseteq [m]} $ encode the sizes of the $ C_t $ and their intersections, and let \gls{allencodings} be the set of all possible values of $ \mathcal{P}(\overline{C}) $.
        We consider $ \overline{C} $ to be unordered, and identify elements in $ \mathcal{P}^*_{m, s, s'} $ which arise from different orderings of the same $ \overline{C} $, thus ensuring that $ \mathcal{P}(\overline{C}) $ is still well-defined.
        Given $ \mathcal{P} \in \mathcal{P}^*_{m, s, s'} $ and a set $ E \subseteq \mathcal{H}_1 $ of size $ |E| = s $, write $C(E, \mathcal{P})$ for the set of all such collections $ \overline{C} $ with $ \bigcup_{t \in [m]} C_t = E $ and $ \mathcal{P}(\overline{C}) = \mathcal{P} $.
        We may now define a function $ \gls{testfnwxbar} \colon\ {\mathcal{H}_1 \choose s} \to \mathbb{R}_{\ge 0} $ by
		$$ \overline{w}_x^{m, s, s'}(E) \coloneqq d_x^{-1} \sum_{\mathcal{P} \in \mathcal{P}^*_{m, s, s'}} \sum_{\overline{C} \in C(E, \mathcal{P})} \sum_{e \in N_x} \prod_{t = 1}^{m} \mathds{1}(C_t \cup e \in \mathcal{G}). $$
        Given a matching $ \mathcal{M}_1 \subseteq \mathcal{H}_1 $ containing $ E $ (but not covering $ x $), if we choose $ e \in N_x \subseteq \mathcal{H}_2 $ randomly, then the function $ \overline{w}_x^{m, s, s'}(E) $ represents the expected number of collections $ \overline{C} \in \bigcup_{\mathcal{P} \in \mathcal{P}^*_{m, s, s'}} C(E, \mathcal{P}) $ for which $ C_t \cup e $ is a conflict from $ \mathcal{G} $ for every $ t \in [m] $.
        As such $ \overline{w}_x^{m, s, s'}(\mathcal{M}_1) $ represents the expected number of $ m $-sets of $ \mathcal{H}_1 $-parts $ C_1, \ldots, C_m $ present in $ \mathcal{M}_1 $ such that $ \sum_{t \in [m]} |C_t| = s' $ but $ |\bigcup_{t \in [m]} C_t| = s $ and each part individually forms a conflict with the randomly chosen $ \mathcal{H}_2 $-edge $ e \in N_x $.
        We are interested in the value of $ \overline{w}_x^{m, s, s'}(\mathcal{M}_1) $, but the function $ \overline{w}_x^{m, s, s'} $ itself may not technically be trackable, so we define further
		$$ \gls{testfnwxmss}(E) \coloneqq \mathds{1}(E \text{ is testable}) \overline{w}_x^{m, s, s'}(E), $$
        which approximates $ \overline{w}_x^{m, s, s'} $ and is trackable, as we will show.}
        
        \subsubsection{Checking trackability}

        We again define a normalising constant $ \beta_x $, analogously to $ \alpha_x $, by
		$ \gls{betax} \coloneqq \max \{ \beta_x', \beta_x'' \}, $
		where
		$$ \beta_x' \coloneqq \max_{j_1 \in [\ell]} \max_{j' \in [j_1 - 1]} d^{j' - j_1} \Delta_{j', 1}(\mathcal{G}_x^{(j_1, 1)}),
		\text{ and }
		\beta_x'' \coloneqq d_x^{-1} \max_{j_1 \in [\ell]} \Delta_{j_1, 0}(\mathcal{G}_x^{(j_1, 1)}). $$
        Note again that $ \beta_x' \le d^{-\varepsilon \new{/3}} $ by (E\ref{cond_mixed_general_codegree_j21}).
        Likewise, observe that $ A(E) = d_x^{-1} $ for any $ E \in \mathcal{G}_x^{(j_1, 1)} $, so $ \Delta^A_{j_1, 0}(\mathcal{G}_x^{(j_1, 1)}) = d_x^{-1} \Delta_{j_1, 0}(\mathcal{G}_x^{(j_1, 1)})$, from which it follows that $ \beta_x'' \le d^{-\varepsilon \new{/3}} $ by the $ j' = j_1 $ case of (E\ref{cond_mixed_general_codegree1}).
        \new{We now fix $ m, s, s' $ and aim to prove that the function $\beta_x^{-1} w_x^{m, s, s'} $ is either $ (d, \varepsilon / \new{10}, \mathcal{C}) $-trackable or $ (d, \varepsilon / \new{10}, \mathcal{C}) $-\new{semi-}trackable.
        By the definition of $ \beta_x'' $, we have that 
        $$
            \overline{w}_x^{m, s, s'}(E) \le d_x^{-1} \sum_{\mathcal{P} \in \mathcal{P}^*_{m, s, s'}} \sum_{\overline{C} \in C(E, \mathcal{P})} d_{\mathcal{G}_x^{(|C_1|, 1)}}(C_1)
            \le d_x^{-1} \Gamma^2 d_x \beta_x''
            \le \beta_x \ell',
        $$
        recalling that $ 1/\ell' \ll 1/ \Gamma \ll 1/i^* $ and thus noting that $ |\mathcal{P}^*_{m, s, s'}| + |C(E, \mathcal{P})| \le \Gamma $.
        We hence see that $ \beta_x^{-1} w_x^{m, s, s'} \le \beta_x^{-1} \overline{w}_x^{m, s, s'} $ is an \new{$ \ell' $-}test function, and observe also that \new{$ \beta_x^{-1} w_x^{m, s, s'} $} satisfies (W\ref{cond_w3}) and (W\ref{cond_w4}) by the definition of testability.}

        \new{Next, we verify that $ \beta_x^{-1} \overline{w}_x^{m, s, s'} $ satisfies (W$\ref{cond_w2}^*$).
        To this end, let $ F \subseteq \mathcal{H}_1 $ be an edge set of size $ j' \in [s - 1] $, and recall that $ \mathcal{Z}_F^{(s)} = \{ Z \in {\mathcal{H}_1 \choose s}: Z \supset F \}$.
        Given $ \mathcal{P} \in \mathcal{P}^*_{m, s, s'} $, let $\mathcal{F}(F, \mathcal{P})$ be the set of all collections $ \overline{F} \coloneqq \{ F_t \}_{t \in [m]} $ of $ m $ subsets $ F_t \subseteq F $ such that $ \bigcup_{t \in [m]} F_t = F $ and there exists some collection $ \overline{C} \subseteq 2^{\mathcal{H}_1} $ with $ \mathcal{P}(\overline{C}) = \mathcal{P} $ and $ C_t \cap F = F_t $ for every $ t \in [m] $.
        Note that any collection $ \overline{C} $ with $ F \subseteq \bigcup \overline{C} $ uniquely determines some such $ \overline{F} \in \mathcal{F}(F, \mathcal{P}(\overline{C})) $.
        Given $ \overline{F} \in \mathcal{F}(F, \mathcal{P}) $ and $ e \in N_x $, we aim to bound $ q(\overline{F}, \mathcal{P}) \coloneqq \sum_{e \in N_x} q(\overline{F}, \mathcal{P}, e) $, where we define $ q(\overline{F}, \mathcal{P}, e) $ as the number of collections $ \overline{C} = \{ C_t \}_{t \in [m]} $ with $ \mathcal{P}(\overline{C}) = \mathcal{P} $ for which $ F_t = C_t \cap F $ and $ C_t \cup e \in \mathcal{G} $ for every $ t \in [m] $.}
        
        \new{Suppose we have already chosen $ C_1, \ldots, C_{t-1} $ such that their intersections are compatible with $ \mathcal{P} $, and for every $ t' \in [t-1] $ we have $ F_{t'} \subseteq C_{t'} $ and $ C_{t'} \cup e \in \mathcal{G} $.
        Write $ C_t' \coloneqq C_1 \cup \cdots \cup C_{t-1} \cup F $ and let $ a_t $ and $ a_t' $ denote the desired sizes of $ C_t \cap C_t' $ and $ C_t \setminus C_t' $, respectively, noting that these are uniquely determined by $ \mathcal{P} $ and $ \overline{F} $.
        If $ a_t, a_t' > 0 $, then to choose $ C_t $ given $ e $, we have at most $ 2^s $ choices for $ C_t \cap C_t' $, each yielding at most $ \Delta_{a_t, 1}(\mathcal{G}^{(a_t + a_t', 1)}) $ choices for $ C_t \setminus C_t' $, so by the definition of $ \beta_x' $ we obtain at most $ 2^s \beta_x' d^{a_t'} $ total choices for $ C_t $.
        If $ a_t = 0 $, then the number of choices for $ C_t $ given $ e $ is at most $ \ell d^{a_t'} $ by (E\ref{cond_mixed_general_degree_j21}).
        If $ a_t' = 0 $, then the number of choices for $ C_t $ is at most $ 2^s $, and this is compatible with most $ \Delta_{a_t, 0}(\mathcal{G}^{(a_t, 1)}) \le \beta_x'' d_x $ choices for $ e $, by the definition of $ \beta_x'' $.
        Note that by definition it is always the case that $ \sum_{t \in [m]} a_t' = s - j' $, and without loss of generality we may assume that $ a_1 \ne 0 $, so that the number of choices for $ (C_1, e) $ is at most $ 2^s \beta_x d^{a_1'} d_x $ in all cases.
        As such we obtain that $ q(\overline{F}, \mathcal{P}) \le \Gamma \beta_x d^{s - j'} d_x $, recalling that $ 1/\Gamma \ll 1/i^* $.
        Observe also that $ |\mathcal{P}^*_{m, s, s'}| + |\mathcal{F}(F, \mathcal{P})| \le \Gamma $.
        We may thus bound
        $$ \overline{w}_x^{m, s, s'}(\mathcal{Z}_F^{(s)}) = d_x^{-1} \sum_{\mathcal{P} \in \mathcal{P}^*_{m, s, s'}} \sum_{\overline{F} \in \mathcal{F}(F, \mathcal{P})} q(\overline{F}, \mathcal{P})
        \le \Gamma^3 \beta_x d^{s - j'}. $$
        Therefore $ \beta_x^{-1} w_x^{m, s, s'}(\mathcal{Z}_F^{(s)}) \le \beta_x^{-1} \overline{w}_x^{m, s, s'}(\mathcal{Z}_F^{(s)}) \le \varepsilon^{-1} d^{s - j'} $, which implies that $ \beta_x^{-1} w_x^{m, s, s'} $ satisfies (W$\ref{cond_w2}^*$).
        Note that it is always the case that at least one of (W\ref{cond_w1}) and (W$\ref{cond_w1*}^*$) holds, so $ \beta_x^{-1} w_x^{m, s, s'} $ is either $ (d, \varepsilon / \new{10}, \mathcal{C}) $-trackable or $ (d, \varepsilon / \new{10}, \mathcal{C}) $-\new{semi-}trackable.}

		\subsubsection{Ignoring untestable sets}

        In this subsection, we again fix $ m, s, s' $, and claim that
        \begin{equation} \label{eqn_wx_wxbar_est}
            \text{if } \quad \overline{w}_x^{m, s, s'}(\mathcal{H}_1) \ge d^{s - \varepsilon^3}, \quad \text{ then } \quad  (1 - d^{-\varepsilon / \new{5}}) \overline{w}_x^{m, s, s'}(\mathcal{H}_1) \le w_x^{m, s, s'}(\mathcal{H}_1) \le \overline{w}_x^{m, s, s'}(\mathcal{H}_1). 
        \end{equation}
        Note that in particular, in this case, we have that $ \beta_x^{-1} w_x^{m, s, s'} $ satisfies (W\ref{cond_w1}), so is trackable.
  
		Indeed, the upper bound is trivial, so we proceed to prove the lower bound.
        Start by writing $ w_x^{m, s, s'}(E) \ge \overline{w}_x^{m, s, s'}(E) - f(E) - g(E) - h(E) $ where 
		\begin{align*} f(E) &\coloneqq \mathds{1}(E \text{ not a matching}) \overline{w}_x^{m, s, s'}(E), \\
		g(E) &\coloneqq \mathds{1}(E \text{ contains a conflict in } \mathcal{C}) \overline{w}_x^{m, s, s'}(E), \\
		 h(E) &\coloneqq \mathds{1}(E \text{ contains a conflict-sharing pair}) \overline{w}_x^{m, s, s'}(E). \end{align*}

        \new{We first bound $ f(\mathcal{H}_1) $ from above by considering the number of choices of $ \overline{C} = \{ C_t \}_{t \in [m]} $ for which $ E = \bigcup_{t \in [m]} C_t $ is not a matching.
        Given $ \mathcal{P} \in \mathcal{P}^*_{m, s, s'} $ and $ e \in N_x $, we aim to bound $ q(\mathcal{P}, e) $, which we define as the number of collections $ \overline{C} = \{ C_t \}_{t \in [m]} $ for which $ C_t \cup e \in \mathcal{G} $ for every $ t \in [m] $, $ \mathcal{P}(\overline{C}) = \mathcal{P} $, and there exist $ e_1 \in C_1, e_2 \in C_2 $ with some vertex $ v \in e_1 \cap e_2 $; note that this case suffices without loss of generality, since we regard $ \overline{C} $ as unordered.}

        \new{Let $ a_1'$ and $ a_2' $ be the sizes of $ C_1$ and $ C_2 \setminus C_1 $, respectively, which are prescribed by $ \mathcal{P} $.
        There are at most $ \ell d^{a_1'} $ choices for $ C_1 $ by (E\ref{cond_mixed_general_degree_j21}), then at most $ a_1' (p+q) $ choices for $ v \in \bigcup C_1 $, $ d $ choices for $ e_2 \in \mathcal{H}_1 $ with $ v \in e_2 $, and $ 2^{a_1'} \beta_x' d^{a_2' - 1} $ choices for $ C_2 \subseteq \mathcal{H}_1 $ with $ e_2 \in C_2 $ and $ C_2 \cup e \in \mathcal{G} $.\footnote{Note that $ e_2 \not \in C_1 $, since we assumed that all conflicts in $ \mathcal{D} $ are matchings.}}
        
        \new{As in the previous section, suppose that $ 3 \le t \le m $ and we have already chosen $ C_1, \ldots, C_{t-1} $ such that their intersections are compatible with $ \mathcal{P} $.
        Similarly, let $ a_t $ and $ a_t' $ be the desired sizes of $ C_t \cap (C_1 \cup \cdots \cup C_{t-1}) $ and $ C_t \setminus (C_1 \cup \cdots \cup C_{t-1}) $, respectively, as determined by $ \mathcal{P} $.
        As before, we see that the number of choices for $ C_t $ is at most $ 2^s \Delta_{a_t, 1}(\mathcal{G}^{(a_t + a_t', 1)}) \le 2^s \ell d^{a_t'} $, in all cases, by (E\ref{cond_mixed_general_degree_j21}) and (E\ref{cond_mixed_general_codegree_j21}).
        Hence, noting that $ \sum_{t \in [m]} a_t' = s $, we obtain that $ q(\mathcal{P}, e) \le \Gamma \beta_x' d^s $.
        This in particular means that
        $$ f(\mathcal{H}_1) = d_x^{-1} \sum_{\mathcal{P} \in \mathcal{P}^*_{m, s, s'}} \sum_{e \in N_x} q(\mathcal{P}, e)
        \le \Gamma^2 \beta_x' d^{s}. $$
        Therefore, since $ \beta_x \le d^{-\varepsilon \new{/3}} $, and by the assumption that $ \overline{w}_x^{m, s, s'}(\mathcal{H}_1) \ge d^{s - \varepsilon^3} $, we obtain that
        $$ f(\mathcal{H}_1) \le d^{- \varepsilon / \new{4}} \overline{w}_x^{m, s, s'}(\mathcal{H}_1). $$}

        By the very same argument, but using instead the degree conditions on $ \mathcal{C} $ and the bound from (\ref{eqn_bound_badpairs}) on the number of $ \varepsilon / 2 $-conflict-sharing pairs $ \{ e, f \} $ given any fixed edge $ e \in \mathcal{H}_1 $, respectively, we obtain that
		$ g(\mathcal{H}_1) \le d^{-\varepsilon / 2} \overline{w}_x^{m, s, s'}(\mathcal{H}_1) $
		and
		$ h(\mathcal{H}_1) \le d^{-\varepsilon / 3} \overline{w}_x^{m, s, s'}(\mathcal{H}_1). $
		Hence the lower bound in (\ref{eqn_wx_wxbar_est}) follows by subtracting these three bounds.

        \subsubsection{Useful estimates}

        \new{We finish this section by computing two bounds which we will need in Section~\ref{section_restrict_safe}.
        Firstly, we bound the order of magnitude of our test functions, specifically
        \begin{equation} \label{eqn_j21_wx_upperbound}
            \overline{w}_x^{m, s, s'}(\mathcal{H}_1) \le \Gamma^2 d^{s}.
        \end{equation}
        Indeed, we use a similar counting argument to the previous sections.
        Fix $ \mathcal{P} \in \mathcal{P}^*_{m, s, s'} $ and $ e \in N_x $, and attempt to count the number of collections $ \overline{C} = \{ C_t \}_{t \in [m]} $ with $ \mathcal{P}(\overline{C}) = \mathcal{P} $ and $ C_t \cup e \in \mathcal{G} $ for every $ t \in [m] $.
        For each $ t \in [m] $, let $ a_t $ and $ a_t' $ be the sizes of $ C_t \cap (C_1 \cup \cdots \cup C_{t-1}) $ and $ C_t \setminus (C_1 \cup \cdots \cup C_{t-1}) $, respectively, as prescribed by $ \mathcal{P} $.
        Suppose now that we have chosen sets $ C_1, \ldots, C_{t-1} $ compatible with $ \mathcal{P} $.
        If $ a_t = 0 $ then we have at most $ \ell d^{a_t'} $ choices for $ C_t $ by (E\ref{cond_mixed_general_degree_j21}).
        If $ a_t' = 0 $ then we have at most $ 2^s $ choices for $ C_t $.
        Otherwise we have at most $ 2^s d^{a_t' - \varepsilon \new{/3}} $ choices for $ C_t $ by (E\ref{cond_mixed_general_codegree_j21}).
        Since $ \sum_{t \in [m]} a_t' = s $ by definition, we obtain in total at most $ \Gamma d^s $ choices for the collection $ \overline{C} $.
        Thus we may bound
        $$ \overline{w}_x^{m, s, s'}(\mathcal{H}_1) \le d_x^{-1} \sum_{\mathcal{P} \in \mathcal{P}^*_{m, s, s'}} \sum_{e \in N_x} \Gamma d^{s} \le \Gamma^2 d^s, $$
        as required, recalling that $ 1/\Gamma \ll 1/i^* $.}
        
        \new{Secondly, given $ m \in [i^*] $ and $ s \in [m \ell - 1] $, we claim further that
        \begin{equation} \label{eqn_j21_wx_nondisjbound}
            \hat{w}_x^{m, s}(E) \coloneqq \sum_{s' = s + 1}^{m \ell} \overline{w}_x^{m, s, s'}(E) \le d^{s - \varepsilon / \new{4}}.
        \end{equation}
        Indeed, consider following the argument above when $ s < s' $, in which case there must exist some $ t \in [m] $ with $ a_t \ne 0 \ne a_t' $.
        We therefore obtain the improved bound
        $$ \overline{w}_x^{m, s, s'}(\mathcal{H}_1) \le \Gamma^2 d^{s - \varepsilon \new{/3}}, $$
        from which the desired result follows.}

	\subsection{Constructing the matching}
	\label{section_proof_construction}
	
		We are now ready to construct the matching in two stages, as described previously.
		
		\subsubsection{Obtaining $ \mathcal{M}_1 $}
        
		We start by applying Theorem~\ref{theorem_cfhm_mod} to the hypergraph $ \mathcal{H}_1' $ defined in Section~\ref{section_proof_regularising}.
		Since $ \mathcal{H}_1' $ may have many more vertices than $ \mathcal{H}_1 $, we need to ensure that only are a small proportion of the vertices \new{in $ P $} are left uncovered.
		To do this, simply define another 1-uniform test function \new{(for $ \mathcal{H}_1'$)}
		$$ \gls{testfnw1}(e) \coloneqq \mathds{1}(e \in \mathcal{H}_1).$$
        For any matching $ \mathcal{N} \subseteq \mathcal{H}_1' $, the value $ p w_1(\mathcal{N}) $ is equal to the number of vertices of $ P $ covered by $ \mathcal{N} $.
		Note that $ w_1 $ is a $ (d, \varepsilon / \new{10}, \mathcal{C}) $-trackable \new{$ \ell' $-}test function since \new{(H\ref{cond_h_h1degree}) and the fact that $ |P| \ge d^{\varepsilon} $ (recalling (S\ref{cond:setup_Psize}))} imply that $w_1(\mathcal{H}_1)=|\mathcal{H}_1|>|P|d/2p> d^{1+\varepsilon/2}$, and (W\ref{cond_w2})--(W\ref{cond_w4}) hold trivially.
		Observe also that for every other test function $ w $ that we have defined previously on subsets of $ \mathcal{H}_1 $, we may extend $ w $ to subsets of $ \mathcal{H}_1' $ by setting $ w(E) = 0 $ for all $ E $ with $ E \not \subseteq \mathcal{H}_1 $, without affecting whether $ w $ is trackable or \new{semi-}trackable.
		
		We may therefore apply Theorem~\ref{theorem_cfhm_mod} to the hypergraph $ \mathcal{H}_1' $, the set of all test functions we have defined, and the conflict hypergraph $ \mathcal{C} $, with $ \ell' $ in place of $ \ell $, to obtain a matching $ \mathcal{M}_1' \subseteq \mathcal{H}_1' $.
		\new{Since all of our test functions are zero on $ \mathcal{H}_1' \setminus \mathcal{H}_1 $,} this induces a $ \mathcal{C} $-free matching $ \mathcal{M}_1 \subseteq \mathcal{H}_1 $ such that, for each of the $ j $-uniform test functions $ w $ which we have defined, we have $ w(\mathcal{M}_1) = (1 \pm d^{-\varepsilon^3}) d^{-j} w(\mathcal{H}_1) $ if $ w $ is trackable, and $ w(\mathcal{M}_1) \le d^{\varepsilon / \new{4}} $ if $ w $ is \new{semi-}trackable.
		In particular, applied to $ w_1 $, this means that $ \new{|\mathcal{M}_1| = |\mathcal{M}_1' \cap \mathcal{H}_1|} \ge (1 - d^{-\varepsilon^3}) d^{-1} |\mathcal{H}_1| \ge (1 - d^{-\varepsilon^3}) (1 - d^{-\varepsilon}) |P| / p $ by the degree condition (H\ref{cond_h_h1degree}).
        Hence at most $ \new{(1 - (1 - d^{-\varepsilon^3}) (1 - d^{-\varepsilon}))|P| \le}  d^{-\varepsilon^3 / 2} |P| $ vertices of $ P $ are left uncovered by $ \mathcal{M}_1 $.
		Write this set as
		$ P' \coloneqq \{ x_1, \ldots, x_M \} $
		for $ M \le d^{-\varepsilon^3}|P| $.

		\subsubsection{Restricting to safe edges}
        \label{section_restrict_safe}

		As discussed, in order to avoid conflicts in the case $ j_2 = 1 $, we restrict to a smaller set of safe edges for each vertex in $ P' $ when choosing $ \mathcal{M}_2 $.
        \new{Recalling that $ N_x = \{ e \in \mathcal{H}_2 : x \in e \} $ and $ d_x = |N_x| $}, we show now that for each $ x \in P' $, there exists a subset $ \gls{safeedges} \subseteq N_x $ of safe edges such that if $ C \cup e \in \mathcal{D} $ is a conflict with $ C \in {\mathcal{M}_1 \choose j_1} $ for some $ j_1 \in [\ell] $ and $ e \in N_x $, then $ e \not \in N_x^s $.

        \new{We use the inclusion-exclusion principle to bound the number of safe edges from below, showing that we can choose $ N_x^s $ of positive density; that is, we ensure that $ |N_x^s| \ge \lambda d_x $, for some constant $ \lambda > 0 $ to be specified, so that restricting to $ N_x^s $ does not significantly limit our choice of edges for $ \mathcal{M}_2 $.
        To do this, we need to estimate the number of $ m $-sets $ \{ C_1, \ldots, C_{m} \} $ such that each $ C_{t} \subseteq \mathcal{M}_1 $ forms a conflict with the same edge $ e $.
        This is possible because $ \mathcal{M}_1 $ is sufficiently pseudorandom, in an appropriate sense, which we show using the test functions $ w_x^{m, s, s'} $, defined in Section~\ref{section_matching_proof_j21}.}
        For the remainder of this section, we again write $ \mathcal{G} = \bigcup_{j \in [\ell]} \mathcal{D}^{(j, 1)} $ for \new{simplicity, recalling our notation from Section~\ref{section:simply_bounded}.}
		Fix $ x \in P' $, and recall that $ i^* \in \mathbb{N} $ satisfies $ 1/i^* \ll 1 / \ell $.
		
        For each $ C \subseteq \mathcal{H}_1 $, start by defining
		$ B_C \coloneqq \{ e \in N_x : C \cup e \in \mathcal{G} \} $
		to be the set of edges in $ \mathcal{H}_2 $ containing $ x $ which complete a conflict with $ C $, and for each $ m \in [i^*] $, \new{recalling the definition of $ E_i(\cdot) $ from Section~\ref{section_proof_conflicts}}, set
		$$ a_m \coloneqq \sum_{e \in N_x} {|E_1(\mathcal{G}_{[\emptyset, e]}) \cap 2^{\mathcal{M}_1}| \choose m} $$
		as the sum over all $ e \in N_x $ of the numbers of $ m $-sets of partial conflicts (of any sizes) in $ \mathcal{M}_1 $, all of which form a conflict with $ e $.
		Define
		$ N_x^s \coloneqq \{ e \in N_x : E_1(\mathcal{G}_{[\emptyset, e]}) \cap 2^{\mathcal{M}_1} = \emptyset \} $
		to be the set of safe edges.
		We may assume that $ i^* $ is chosen to be odd, so by the inclusion-exclusion principle, we obtain that
		\begin{equation} \label{eqn_j21_ex_lb}
			|N_x^s| = \left| N_x \setminus \bigcup_{C \subseteq \mathcal{M}_1} B_C \right| \ge d_x - a_1 + a_2 - a_3 + \cdots - a_{i^*}.
		\end{equation}
		We now aim to show that this is suitably close to an exponential series, which guarantees that a constant proportion of the edges of $ N_x $ belong to $ N_x^s $.

		\new{Fix $ m \in [i^*] $ then, considering all of the possible sizes and intersections of a set of $ m $ conflicts, we see by the definition of $ w_x^{m, s, s'} $ in Section~\ref{section_j21_testdef} that
        \begin{equation} \label{eqn_am_wxi}
            a_m = d_x \sum_{s' \in [m \ell]} \sum_{s \in [s']} w_x^{m, s, s'}(\mathcal{M}_1).
        \end{equation}}
        We first claim that
        \begin{equation} \label{eqn_j21_wx_h1_m1}
            w_x^{m, s, s'}(\mathcal{M}_1) = (1 \pm d^{-\varepsilon^3 / 2}) d^{-s} \overline{w}_x^{m, s, s'}(\mathcal{H}_1) \pm d^{-\varepsilon^3 / 2}
        \end{equation}
        for every $ s' \in [m \ell] $ and $ s \in [s'] $.
        Indeed, we must consider three cases.
        Firstly, if $ \beta_x^{-1} w_x^{m, s, s'} $ is \new{semi-}trackable, then Theorem~\ref{theorem_cfhm_mod} tells us that $ w_x^{m, s, s'}(\mathcal{M}_1) \le \beta_x d^{\varepsilon / \new{4}} \le d^{-\varepsilon / \new{12}} $, \new{recalling that $ \beta_x \le d^{-\varepsilon / 3} $}.
        Secondly, if $ \beta_x^{-1} w_x^{m, s, s'} $ is trackable but $ \overline{w}_x^{m, s, s'}(\mathcal{H}_1) < d^{s - \varepsilon^3} $, then we obtain that $ w_x^{m, s, s'}(\mathcal{M}_1) = (1 \pm d^{-\varepsilon^3}) d^{-s} w_x^{m, s, s'}(\mathcal{H}_1) \le 2 d^{-s} \overline{w}_x^{m, s, s'}(\mathcal{H}_1) \le 2 d^{- \varepsilon^3} $.
        Since $ d^{-s} \overline{w}_x^{m, s, s'}(\mathcal{H}_1) \le d^{-\varepsilon^3} $ in both of these cases, it is certainly true that $ w_x^{m, s, s'}(\mathcal{M}_1) = d^{-s} \overline{w}_x^{m, s, s'}(\mathcal{H}_1) \pm  d^{-\varepsilon^3 / 2} $.
        Finally, if $ \beta_x^{-1} w_x^{m, s, s'} $ is trackable and $ \overline{w}_x^{m, s, s'}(\mathcal{H}_1) \ge d^{s - \varepsilon^3} $, then the estimate from Theorem~\ref{theorem_cfhm_mod}, combined with (\ref{eqn_wx_wxbar_est}), gives us that
        $ w_x^{m, s, s'}(\mathcal{M}_1) = (1 \pm d^{-\varepsilon^3 / 2}) d^{-s} \overline{w}_x^{m, s, s'}(\mathcal{H}_1) $,
        so \eqref{eqn_j21_wx_h1_m1} holds in all cases.
        \new{Plugging \eqref{eqn_j21_wx_h1_m1} into \eqref{eqn_am_wxi} and recalling \eqref{eqn_j21_wx_nondisjbound}, we obtain that
        \begin{align*}
            a_m &= (1 \pm d^{-\varepsilon^3 / 2}) d_x \sum_{s' \in [m \ell]} \sum_{s \in [s']} d^{-s} \overline{w}_x^{m, s, s'}(\mathcal{H}_1) \pm d^{-\varepsilon^3 / 2} d_x\\
                &= (1 \pm d^{-\varepsilon^3 / 2}) d_x \sum_{s' \in [m \ell]} d^{-s'} \overline{w}_x^{m, s', s'}(\mathcal{H}_1) \pm d^{-\varepsilon^3 / 3} d_x. \stepcounter{equation}\tag{\theequation}\label{eqn_j21_am_wsprimes}
        \end{align*}}
        
		In order to approximate this value, for each $ e \in N_x $ and $ j \in [\ell] $, define
		$ \gamma_{e, j} \coloneqq d^{-j} |\mathcal{G}_{[\emptyset, e]}^{(j, 1)}| \le \ell $
		(by (E\ref{cond_mixed_general_degree_j21})).
        Then set
		$ \gamma_e \coloneqq \sum_{j \in \ell} \gamma_{e, j} \le \ell^2, $
		\new{as an estimate for} the total number of partial conflicts completed by $ e $, of any size, which we expect to appear in $ \mathcal{M}_1 $.
        \new{For each $ s' \in [m \ell] $, write $ \mathcal{Q}_{m, s'} $ for the set of sequences $ \mathbf{b} = (b_j)_{j \in [\ell]} $ with $ \sum_{j \in [\ell]} b_j = m $ and $ \sum_{j \in [\ell]} j b_j = s' $.
		Considering the number of ways to choose $ m $ conflicts containing $ e $ with the sum of sizes of $ \mathcal{H}_1 $-parts equal to $ s' $, we see that we may write
        \begin{equation} \label{eqn_wxibar_gamma}
            \sum_{s \in [s']} d_x \overline{w}_x^{m, s, s'}(\mathcal{H}_1)
            = \sum_{\mathbf{b} \in \mathcal{Q}_{m, s'}} \sum_{e \in N_x} \prod_{j \in [\ell]} {|\mathcal{G}_{[\emptyset, e]}^{(j, 1)}| \choose b_j}
            = d^{s'} \sum_{e \in N_x} \sum_{\mathbf{b} \in \mathcal{Q}_{m, s'}} \prod_{j \in [\ell]} \frac{\gamma_{e, j}^{b_j}}{b_j!} \pm d^{s'-1/2} d_x,
        \end{equation}
        where the error term is obtained using the fact that $ {n \choose k} = n^{k} / k! \pm f(n) $ for a polynomial $ f $ of degree at most $ k - 1 $, as well as the fact that $ |\mathcal{G}_{[\emptyset, e]}^{(j, 1)}| \le \ell d^j $ by (E\ref{cond_mixed_general_degree_j21}).
        Recalling \eqref{eqn_j21_wx_upperbound}, we see that all summands with $ s < s' $ contribute lower order terms to \eqref{eqn_wxibar_gamma}, so in fact
        \begin{equation} \label{eqn_wxibar_gamma_sprime}
            d_x \overline{w}_x^{m, s', s'}(\mathcal{H}_1) = d^{s'} \sum_{e \in N_x} \sum_{\mathbf{b} \in \mathcal{Q}_{m, s'}} \prod_{j \in [\ell]} \frac{1}{b_j!} \gamma_{e, j}^{b_j} \pm 2 d^{s'-1/2} d_x.
        \end{equation}
        Plugging \eqref{eqn_wxibar_gamma_sprime} into \eqref{eqn_j21_am_wsprimes} gives
        \begin{equation} \label{eqn_j21_am_gammasplit}
            a_m = (1 \pm d^{-\varepsilon^3 / 2}) \sum_{s' \in [m \ell]} \sum_{e \in N_x} \sum_{\mathbf{b} \in \mathcal{Q}_{m, s'}} \prod_{j \in [\ell]} \frac{\gamma_{e, j}^{b_j}}{b_j!} \pm d^{-\varepsilon^3 / 4} d_x.
        \end{equation}
        Now observe by the multinomial theorem that
        \begin{equation} \label{eqn_j21_gamma_multinom}
            \gamma_e^m = \left( \sum_{j \in [\ell]} \gamma_{e, j} \right)^{m} = m! \sum_{s' \in [m \ell]} \sum_{\mathbf{b} \in \mathcal{Q}_{m, s'}} \prod_{j \in [\ell]} \frac{\gamma_{e, j}^{b_j}}{b_j!}.
        \end{equation}
        Hence, plugging \eqref{eqn_j21_gamma_multinom} into \eqref{eqn_j21_am_gammasplit}, we obtain the estimate
        \begin{equation} \label{eqn_j21_am_final}
            a_m = (1 \pm d^{-\varepsilon^3 / 2}) \sum_{e \in N_x} \frac{\gamma_e^m}{m!} \pm d^{-\varepsilon^3 / 4} d_x.
        \end{equation}}

        For $ x \in \mathbb{R} $, define $ S(x) = \sum_{m = 0}^{i^*} (-x)^m / m! $.
		Now \new{plugging} \eqref{eqn_j21_am_final} into \eqref{eqn_j21_ex_lb} gives 
		\begin{equation} \label{eqn_ex_s_est}
			|N_x^s| \ge \sum_{e \in N_x} 1 - \gamma_e + \frac{1}{2!} \gamma_e^2 - \cdots - \frac{1}{i^*!} \gamma_e^{i^*} \pm i^* d^{-\varepsilon^3 / 4} d_x
            \ge \sum_{e \in N_x} S(\gamma_e) - d^{-\varepsilon^3 / 5} d_x.
		\end{equation}
		To bound $ \sum_{e \in N_x} S(\gamma_e) $ from below, let $ \delta \coloneqq \exp(-\ell^2) / 3 > 0 $, then recalling that $ 1/i^* \ll 1/\ell $, we see that for every $ 0 \le x \le \ell^2 $, we have
		$ S(x) = \exp(-x) \pm \delta. $
		The choice of such an $ i^* $, depending only upon $ \ell $, is possible by the uniform convergence of the exponential series on the interval $ [0, \ell^2] $.
		Then, since $ 0 \le \gamma_e \le \ell^2 $ for every $ e \in N_x $, we obtain
		$$
			\sum_{e \in N_x} S(\gamma_e) \ge \sum_{e \in N_x} \exp(-\gamma_e) - d_x \delta \ge d_x \exp \left( -\frac{1}{d_x} \sum_{e \in N_x} \gamma_e \right) - d_x \delta \ge d_x (\exp(-\ell^2) - \delta)
		$$
		by convexity of the exponential function and the bound on $ \gamma_e $.
		Substituting this back into (\ref{eqn_ex_s_est}), we see that
		$ |N_x^s| \ge d_x (\exp(-\ell^2) - 2 \delta) \ge d_x \exp(-\ell^2) / 3. $
		Hence we may take $ \lambda = \lambda_{\ell} \coloneqq \exp(-\ell^2) / 3 > 0 $.
		
		\subsubsection{Choosing $ \mathcal{M}_2 $}
		
		We may now proceed to choose the edges of $ \mathcal{M}_2 $ randomly from the sets we have defined.
        Recall that $ P' = \{ x_1, \ldots, x_M \} $ is the subset of $ P $ not yet covered by $ \mathcal{M}_1 $, and define $ \mathcal{H}_2' \coloneqq \mathcal{H}_2[P' \cup R] $.
		For each $ i \in [M] $, choose an edge $ e_i $ uniformly at random from the set $ N^s_{x_i} $, so each possible edge $ e \in N^s_{x_i} $ is taken with probability
		\begin{equation} \label{eqn_edge_prob}
			\mathbb{P}[e_i = e] \le (\lambda d_{\mathcal{H}_2}(x_i))^{-1}.
		\end{equation}
		Recall also that (H$\ref{cond_h_h2d'}'$) and (H$\ref{cond_h_h2cd'}'$) together imply that $ d_{\mathcal{H}_2}(x_i) \ge d^{\varepsilon} $, so $ |N^s_{x_i}| \ge \lambda d^{\varepsilon} > 0 $.
		Set $ \mathcal{M}_2 \coloneqq \{ e_1, \ldots, e_m \} $ and take
		$ \mathcal{M} \coloneqq \mathcal{M}_1 \cup \mathcal{M}_2 $
		to be the combination of our matching in $ \mathcal{H}_1 $ with these new edges chosen from $ \mathcal{H}_2 $.
		We now use Lemma~\ref{lemma_local_mid} to ensure that $ \mathcal{M} $ is $ \mathcal{D} $-free; recall that by excluding the conflicts added in Section~\ref{section_proof_matching}, this will also ensure that $ \mathcal{M}_2 $ is a matching.
		
		For each $ j_2 \in [2, \ell] $, define bad events
		$ B_D \coloneqq \{ D \subseteq \mathcal{M}_2 \} $
		for each $ D \in {\mathcal{H}_2' \choose j_2} $ which could appear as the $ \mathcal{H}_2 $-part of a conflict in $ \mathcal{D}^{(j_1, j_2)} $ for some $ j_1 \in [0, \ell] $, that is $ D \in E_2(\mathcal{D}_{[C, \emptyset]}^{(j_1, j_2)}) $ for some $ C \in {\mathcal{M}_1 \choose j_1} $.\footnote{\new{Recall here our notation from Sections \ref{section:simply_bounded}, \ref{section_general_conditions}, and \ref{section_proof_conflicts}.}}
		Let $ \mathfrak{A} $ be the set of all such events\new{, for any $ (j_1, j_2) $}, which we aim to avoid.
        Note that, for each event $ B_D \in \mathfrak{A} $, we have $ \mathbb{P}[B_D] \le \lambda^{-1} d^{-\varepsilon} <  1/2 $, so in order to apply Lemma~\ref{lemma_local_mid}, it suffices to show that we also have
		$ \sum_{A \in \mathfrak{B}(D)} \mathbb{P}[A] \le 1 / 4, $
		where $ \mathfrak{B}(D) \coloneqq \{ B_{D'} \in \mathfrak{A} : D' \in E_2(\mathcal{D}_x) \text{ for some } x \in V_P(D) \} $, \new{recalling the definition of $ V_P(\cdot) $ from Section~\ref{section_general_conditions}}; since the event $ B_D $ depends only upon the edge choices for the vertices $ x \in V_P(D) $, it is mutually independent from the set of all events $ B_{D'} $ for which $ V_P(D) \cap V_P(D') = \emptyset $.
		
		Now fix some $ j_1 \in [0, \ell] $, $ j_2 \in [2, \ell] $, and $ x \in P' $, and again write $ \mathcal{G} \coloneqq \mathcal{D}^{(j_1, j_2)} $ for ease of notation.
		Further write
		$ \mathcal{G}^{\mathcal{M}_1} \coloneqq \{ C \cup D \in \mathcal{G} : C \subseteq \mathcal{M}_1 \text{ and } D \subseteq \mathcal{H}_2' \}. $
		Because $ |V_P(D)| = j_2 \le \ell $, and there are at most $ \ell^2 $ choices for $ j_1 $ and $ j_2 $, it is (more than) sufficient to show that
		\begin{equation} \label{eqn_probsum_nts}
			\sum_{D \in E_2(\mathcal{G}^{\mathcal{M}_1}_x)} \mathbb{P}[D \subseteq \mathcal{M}_2] \le d^{-\varepsilon^3 / 3}.
		\end{equation}
		 
		Given $ C \in {\mathcal{H}_1 \choose j_1} $, say that an $ \mathcal{H}_2 $-part \textit{$ D \in E_2(\mathcal{G}_{[C], x}) $ is blocked} if $ V_P(D) \not \subseteq P' $, that is some vertex $ y \in V_P(D) $ is already covered by $ \mathcal{M}_1 $.
        Note that such conflicts can be ignored as they will never be present in $ \mathcal{M} $, because no edge containing $ y $ is chosen in $ \mathcal{M}_2 $; the remainder of the proof is therefore concerned with bounding \textit{unblocked} conflicts.
        We define
		$ \mathcal{B}(C, x) \coloneqq \{ D \in E_2(\mathcal{G}_{[C], x}) : D \text{ is blocked} \} $
		and
		$ \mathcal{U}(C, x) \coloneqq \{ D \in E_2(\mathcal{G}_{[C], x}) : D \text{ is not blocked} \}. $
		Then by definition
		\begin{equation} \label{eqn_e2bard_char}
			E_2(\mathcal{G}^{\mathcal{M}_1}_x) = \bigcup_{C \in {\mathcal{M}_1 \choose j_1}} \mathcal{U}(C, x).
		\end{equation}
		Furthermore, observe that by (\ref{eqn_edge_prob}), and the independence of edge choices for distinct vertices of $ P' $, we have
		\begin{equation} \label{eqn_conflict_prob}
			\mathbb{P}[D \subseteq \mathcal{M}_2] \le \prod_{y \in V_P(D)} (\lambda d_{\mathcal{H}_2}(y))^{-1} = \lambda^{-j_2} A(D)
		\end{equation}
		for any (unblocked) $ D \in {\mathcal{H}_2' \choose j_2} $.
		Hence to show (\ref{eqn_probsum_nts}), using (\ref{eqn_e2bard_char}) and (\ref{eqn_conflict_prob}), it is enough to prove that
		\begin{equation} \label{eqn_probsum_nts2}
			\sum_{C \in {\mathcal{M}_1 \choose j_1}} A(\mathcal{U}(C, x)) \le d^{-\varepsilon^3 / 3}.
		\end{equation}
		We now prove (\ref{eqn_probsum_nts2}) using the test functions we have defined, as discussed previously.

		\subsubsection{Bounding unblocked conflicts}
		\label{section_bound_unblocked_conflicts}
		
		We consider first the case that the function $ \alpha_x^{-1} w_x $ is $ (d, \varepsilon / \new{10}, \mathcal{C}) $-\new{semi-}trackable, so Theorem~\ref{theorem_cfhm_mod} tells us that
		$ \alpha_x^{-1} w_x(\mathcal{M}_1) \le d^{\varepsilon / \new{4}}. $
		By the definition of $ w_x $ and the bound $ \alpha_x \le d^{-\varepsilon \new{/3}} $, this means that
		$$ \sum_{C \in {\mathcal{M}_1 \choose j_1}} A(\mathcal{G}_{[C], x}) = w_x(\mathcal{M}_1) \le \alpha_x d^{\varepsilon / \new{4}} \le d^{-\varepsilon / \new{12}}, $$		
		which in particular is sufficient for (\ref{eqn_probsum_nts2}).
		
		Assume instead that both of the functions $ \alpha_x^{-1} w_x $ and $ \alpha_x^{-1} w_x' $ are $ (d, \varepsilon / \new{10}, \mathcal{C}) $-trackable.
		In order to obtain (\ref{eqn_probsum_nts2}), we rewrite
		\begin{equation} \label{eqn_unblocked_difference}
			\sum_{C \in {\mathcal{M}_1 \choose j_1}} A(\mathcal{U}(C, x)) = \sum_{C \in {\mathcal{M}_1 \choose j_1}} A(\mathcal{G}_{[C], x}) - \sum_{C \in {\mathcal{M}_1 \choose j_1}} A(\mathcal{B}(C, x)),
		\end{equation}
		and use $ w_x $ and $ w_x' $ to estimate the two sums respectively.
		
		Firstly, by definition and the conclusion of Theorem~\ref{theorem_cfhm_mod}, we have that
		\begin{equation} \label{eqn_totalconflictsm1_upper}
			\sum_{C \in {\mathcal{M}_1 \choose j_1}} A(\mathcal{G}_{[C], x}) = w_x(\mathcal{M}_1) = (1 \pm d^{-\varepsilon^3}) d^{-j_1} w_x(\mathcal{H}_1) \le (1 + d^{-\varepsilon^3}) d^{-j_1} A(\mathcal{G}_x).
		\end{equation}

        To obtain a lower bound for the second term in (\ref{eqn_unblocked_difference}), start by observing that, given $ D \in E_2(\mathcal{G}_x) $, we have $ |V_P(D) \setminus P'| \le (j_2 - 1) \mathds{1}(V_P(D) \setminus P' \ne \emptyset) = (j_2 - 1) \mathds{1}(D \text{ blocked}) $.
        Therefore,
        \begin{align*}
            \sum_{C \in {\mathcal{M}_1 \choose j_1}} A(\mathcal{B}(C, x)) &= \sum_{C \in {\mathcal{M}_1 \choose j_1}} \sum_{D \in E_2(\mathcal{G}_{[C], x})} A(D) \mathds{1}(D \text{ blocked})\\
                &\ge \frac{1}{j_2 - 1} \sum_{C \in {\mathcal{M}_1 \choose j_1}} \sum_{D \in E_2(\mathcal{G}_{[C], x})} A(D) |V_P(D) \setminus P'|\\
                &= \frac{1}{j_2 - 1} \sum_{C \in {\mathcal{M}_1 \choose j_1}} \sum_{D \in E_2(\mathcal{G}_{[C], x})} A(D) \sum_{y \in V_P(D) \setminus \{ x \}} \sum_{e \in \mathcal{M}_1} \mathds{1}(y \in e) \\
                &= \frac{1}{j_2 - 1} w_x'(\mathcal{M}_1). \tag{\theequation} \label{eqn_blocked_wx'}
        \end{align*}
        To see the penultimate equality, note that $ y \in P \setminus P' $ if and only if there is exactly one edge $ e \in \mathcal{M}_1 $ containing $ y $.
        For the final equality, recall the definition of $ w_x' $ in Section~\ref{section_deftestfn}, and note further that if $ C \cup e \subseteq \mathcal{M}_1 $, then $ C \cup e $ cannot contain any conflict-sharing pair or conflict from $ \mathcal{C} $.

		Now by Theorem~\ref{theorem_cfhm_mod} and the estimate (\ref{eqn_wx'_bounds}), we see that
		\begin{equation} \label{eqn_wx'm1_estimate}
			w_x'(\mathcal{M}_1) = (1 \pm d^{-\varepsilon^3}) d^{-j_1 - 1} w_x'(\mathcal{H}_1) \ge (j_2 - 1) (1 - d^{-\varepsilon^3 / 2}) d^{-j_1} w_x(\mathcal{H}_1).
		\end{equation}
		Hence combining (\ref{eqn_blocked_wx'}) and (\ref{eqn_wx'm1_estimate}) we obtain the bound 
		\begin{equation} \label{eqn_blockedconflictsm1_lower}
			\sum_{C \in {\mathcal{M}_1 \choose j_1}} A(\mathcal{B}(C, x)) \ge (1 - d^{-\varepsilon^3 / 2}) d^{-j_1} A(\mathcal{G}_x).
		\end{equation}
		
		We finish by \new{plugging} (\ref{eqn_blockedconflictsm1_lower}) and (\ref{eqn_totalconflictsm1_upper}) back into (\ref{eqn_unblocked_difference}) to obtain
		$$ \sum_{C \in {\mathcal{M}_1 \choose j_1}} A(\mathcal{U}(C, x)) \le 2 d^{-\varepsilon^3 / 2} d^{-j_1} A(\mathcal{G}_x) \le d^{-\varepsilon^3 / 3}, $$
		by (E\ref{cond_mixed_general_degree}), as required for (\ref{eqn_probsum_nts2}).
		Note that all of the arguments made here still work for $ j_1 = 0 $, when $ w_x $ is defined simply on the empty set.

\section{Acknowledgments.}
We are grateful to Michelle Delcourt and Luke Postle for helpful comments and for informing us about some references.
We would also like to thank the referees for their extensive and helpful comments.

\printglossaries

\markboth{}{}

\bibliographystyle{amsplain}
\bibliography{sources}

\appendix

\section{Proofs of Applications} \label{section:appendix}

    In what follows we provide proofs for Theorems~\ref{theorem_cycles}, \ref{theorem_covering}, and~\ref{theorem_covering_application}, as well as further simplification for the determination of the Erd\H{o}s-Gy\'arf\'as function.

    \subsection{Proof of Theorem~\ref{theorem_cycles}}
		Note that in this subsection, $ \ell $ and $ k $ refer to the length and uniformity, respectively, of the tight cycle $ C_{\ell}^k $; we do not make any further reference to the parameters in Theorem~\ref{theorem_matching} of the same names, assuming throughout that $ \ell $ from the theorem has been chosen as a sufficiently large constant (depending on the length of the cycle, but not on~$ n $).
        Let $ \mathcal{T}_1 $ and $ \mathcal{T}_2 $ be disjoint sets of colours of sizes $ t_1 = n / (\ell - k) $ and $ t_2 = n^{1 - \delta} $ respectively, for $ \delta > 0 $ sufficiently small (as specified later).
        We refer to the vertices and edges of $ \mathcal{H} $ as \textit{auxiliary} to distinguish them from the \textit{vertices} and \textit{edges} of the underlying $ K_n^k $.

		\subsubsection{Construction of $ \mathcal{H}_1 $ and $ \mathcal{H}_2 $}
        Let $ (U_{\alpha})_{\alpha \in \mathcal{T}_1 \cup \mathcal{T}_2} $ be a set of $ t_1 + t_2 $ pairwise disjoint copies of $ E(K_n^{k - 1}) $, then set $ P \coloneqq E(K_n^k) $,
        $ Q \coloneqq \bigcup_{\alpha \in \mathcal{T}_1} U_{\alpha} $,
        and $ R \coloneqq \bigcup_{\alpha \in \mathcal{T}_2} U_{\alpha} $.
        We may define a hypergraph $ \mathcal{H}_1 $ by constructing an auxiliary edge corresponding to each pair $ (X, \alpha) $ where $ X $ is a copy of $ K_{\ell - 1}^k $ in $ K_{n}^k $ and $ \alpha \in \mathcal{T}_1 $; formally, we take the union of $ E(X) \subseteq P $ with the copy of $ {V(X) \choose k - 1} $ in $ U_{\alpha} $, so $ p = {\ell - 1 \choose k} $ and $ q = {\ell - 1 \choose k - 1} $.
        Such an auxiliary edge $ (X, \alpha) $ corresponds to a colouring of $ E(X) $ by the colour $ \alpha $, so a matching $ \mathcal{M}_1 \subseteq \mathcal{H}_1 $ yields a well-defined partial colouring $ c_1 $ (that is, a colouring of a subhypergraph) of $ E(K_n^k) $, in which every colour class is a set of copies of $ K_{\ell - 1}^k $ with pairwise intersections of size at most $ k - 2 $; in particular, this means that consecutive edges of a tight cycle cannot belong to distinct blocks within the same colour class.
        Similarly construct auxiliary edges of $ \mathcal{H}_2 $ corresponding to each pair $ (e, \alpha) $ where $ e \in P $ and $ \alpha \in \mathcal{T}_2 $; formally, such an auxiliary edge is represented by the union of $ \{ e \} $ with the copy of $ {e \choose k - 1} $ in $ U_{\alpha} $, so $ r = k $.
        
		To check the degree conditions for $ \mathcal{H}_1 $, fix some auxiliary vertex $ x \in P \cup Q $ and count the number of auxiliary edges $ (X, \alpha) $ to which it belongs.
        If $ x \in P $ is a $ k $-set, then there are $ {n - k \choose \ell - 1 - k} $ possible choices for the remaining vertices of $ X $ and $ t_1 = n / (\ell - k) $ choices for the colour $ \alpha $.
        If instead $ x \in U_{\alpha} $ is a $ (k - 1) $-set, then there are $ {n - k + 1 \choose \ell - k} $ possible choices for the remaining vertices of $ X $, and $ \alpha $ is already determined.
        Hence in each case the condition (H\ref{cond_h_h1degree}) is satisfied with
		$ d = \frac{n^{\ell - k}}{(\ell - k)!} $.
		Likewise given two auxiliary vertices $ x, y \in P \cup Q $, they determine at least $ k + 1 $ vertices of $ X $ (if both $ x, y \in P $) or at least $ k $ vertices of $ X $ and the colour $ \alpha $, so in each case the codegree $ d(x, y) $ is at most $ n^{\ell - k - 1} $, which is bounded from above by $ d^{1 - \varepsilon} $ for $ \varepsilon $ sufficiently small in terms of $ \ell $ and $ k $, and (H\ref{cond_h_h1codegree}) holds.

        Similarly, $ d_{\mathcal{H}_2}(x) = t_2 = n^{1 - \delta} $ for any $ x \in P $, and $ d_{\mathcal{H}_2}(y) = d_{K_n^k}(y) = n - k + 1 $ for any $ y \in R $, so (H\ref{cond_h_h2degree}) is satisfied provided that $ n^\delta \le d^{\varepsilon^4} $; since $ d $ is bounded from below by a constant power of $ n $, it suffices to choose $ \delta $ sufficiently small in terms of $ \varepsilon $.
        Furthermore we always have $ d(x, y) \le 1 $ so (H\ref{cond_h_h2codegree}) also holds, and clearly $ \exp(d^{\varepsilon^3}) \ge |P \cup Q| $ for $ n $ sufficiently large.
        
        A matching $ \mathcal{M}_2 \subseteq \mathcal{H}_2 $ gives a well-defined partial colouring $ c_2 $ of $ E(K_n^k) $ with colours from~$ \mathcal{T}_2 $, such that distinct edges of any given colour have pairwise intersections of size at most $ k - 2 $.
        If~$ \mathcal{M}_2 $ completes $ \mathcal{M}_1 $ to a $ P $-perfect matching $ \mathcal{M} $, then the union $ c $ of the two colourings gives a well-defined complete colouring of $ E(K^k_n) $.

        \subsubsection{Conflicts in general}
        We now use conflict hypergraphs to avoid any copy $ Z $ of $ C_{\ell}^k $ being coloured with at most $ k $ distinct colours.
        Given a problematic colouring of some such $ Z $, suppose that $ 0 \le t \le k - 1 $ colours appear on exactly one edge of $ Z $, then we may remove all such edges to obtain a subgraph $ Z' \subseteq Z $ with $ \ell - t $ edges, which is coloured by some colouring $ c' $ using at most $ k - t $ distinct colours.
        Observe that it suffices to forbid all such subcolourings $ (Z', c') $, and indeed we must forbid them, since any completion to a colouring of $ Z $ would be problematic, even if every remaining edge received a distinct new colour.
        Suppose therefore that the $ \ell - t $ edges in $ E(Z') \subseteq E(Z) $ receive at most $ k - t $ colours in $ c $ (and assume now that each such colour appears on at least two edges of $ Z' $).
        Recall that $ c $ arises from a ($ P $-perfect) matching of $ \mathcal{H} $, and let $ E $ be the smallest submatching which gives rise to the colouring $ c' $ of $ Z' $.
        Write $ E = C \cup D $, where $ C $ and $ D $ are sets of $ j_1 \in [0, \ell] $ and $ j_2 \in [0, \ell] $ auxiliary edges from $ \mathcal{H}_1 $ and $ \mathcal{H}_2 $ respectively.

        Fix some cyclic ordering of the vertices and edges of $ Z $.
        We may assume further that the edges of $ Z' $ coloured by any given auxiliary edge $ (X, \alpha) \in C $ are consecutive (in this ordering).
        Indeed, suppose that, for some $ 1 \le i_1 < i_2 \le i_3 < \ell $, we have $ E(Z) = \{ e_1, \ldots, e_{i_1}, \ldots, e_{i_2}, \ldots, e_{i_3}, \ldots, e_{\ell} \} $ (in cyclic order) and $ (X, \alpha) $ colours $ e_1, \ldots, e_{i_1} $ and $ e_{i_2}, \ldots, e_{i_3} $ but none of $ e_{i_1 + 1} $ or $ e_{i_3 + 1}, \ldots, e_{\ell} $ (note that there may or may not be other edges from $ e_{i_1 + 2}, \ldots, e_{i_2 - 1} $ coloured by $ (X, \alpha) $).
        Then at least one vertex $ v \in e_{\ell} $ is not contained in $ V(X) $ so, since $ X $ is a clique, none of the $ k $ edges of the cycle $ Z $ containing $ v $ are coloured by $ X $, which means that in fact $ \ell - i_3 \ge k $, and in particular $ i_3 \le \ell - k $.
        Thus we may find a (tight) path $ e_1 f_2 \cdots f_{i_3 - 1} e_{i_3} $, with $ i_3 $ edges, contained entirely in $ X $ (and so disjoint from the edges $ e_{i_3 + 1}, \ldots, e_{\ell} $) such that $ e_1 f_2 \cdots f_{i_3 - 1}  e_{i_3} \cdots e_{\ell} $ is a copy of $ C_{\ell}^k $ coloured with at most $ k $ distinct colours, with strictly fewer auxiliary edges from $ \mathcal{H}_1 $ colouring non-consecutive edges of the cycle.
        We see therefore that it suffices to avoid only submatchings $ E $ in which the edges of $ Z' $ coloured by any given auxiliary edge from $ \mathcal{H}_1 $ are consecutive in the cyclic ordering of $ Z $.
        We will take all possible such $ E $ to be our conflicts.

        In order to prove boundedness conditions for the conflict hypergraphs which we will define, we must bound the total number of such conflicts, given $ j_1 $ and $ j_2 $.
        To this end, we start by making some general observations about the number of vertex and colour choices involved in such conflicts.
        Write $ C = \{ (X_i, \alpha): i \in [j_1] \} $, where the $ X_i $ are enumerated according to our cyclic ordering of $ E(Z) $ (this is well-defined since each $ X_i $ colours only consecutive edges of the cycle).
        Observe that each auxiliary edge $ (X_i, \alpha) \in C $ colours the edges spanned by a set $ V(X_i) \cap V(Z) $ of at least $ k $ vertices on the cycle $ Z $, of which at most the first $ k - 1 $ (in our cyclic ordering of $ V(Z) $) are also contained in $ V(X_{i - 1}) $.
        Letting $ S_i \subseteq V(X_i) \cap V(Z) $ be these first $ k - 1 $ vertices, this means that the sets $ V(X_i) \setminus S_i $ are disjoint.
        Each such $ V(X_i) \setminus S_i $ has size $ \ell - k $, and counts exactly one vertex for each edge on the cycle $ Z $ which is coloured by $ X_i $, that is, $ |(V(X_i) \setminus S_i) \cap V(Z)| = |E(X_i) \cap E(Z)| $.
        Hence the size of the (disjoint) union $ |\bigcup_{i \in [j_1]} V(X_i) \setminus S_i| = j_1 (\ell - k) $ is exactly equal to the number of vertices contained in $ \bigcup_{i \in [j_1]} V(X_i) \setminus V(Z) $ plus the number of edges of $ Z' $ which are coloured by the $ \mathcal{H}_1 $-part $ C $, which is $ \ell - t - j_2 $.
        The number of vertices involved in a conflict is therefore $ j_1 (\ell - k) $ plus the remaining $ \ell - (\ell - t - j_2) = j_2 + t $ uncounted vertices in $ V(Z) $.
        We additionally have $ k - t $ colours, so in total we obtain the following.
        \begin{observation} \label{obs_cycles_conflict}
            Each conflict $ E $ consists of $ j_1 (\ell - k) + j_2 + k $ vertex/colour choices.
        \end{observation}

        Now let $ j' \in [j_1] $, fix a set $ E' \subseteq C $ of $ j' $ auxiliary edges from $ \mathcal{H}_1 $, and suppose that $ E' $ colours $ b $ edges of $ Z' $, which cover $ a $ vertices of $ Z $, using $ m $ distinct colours.
        Counting in a similar way to before, we see that $ E' $ fixes $ j'(\ell - k) + a - b $ vertices and $ m $ colours of the conflict part $ C $.
        Hence subtracting this from Observation~\ref{obs_cycles_conflict}, the total number of vertex/colour choices for a conflict $ E $ containing $ E' $ is $ (j_1 - j') (\ell - k) + j_2 + k - a + b - m $.
        Considering each vertex in cyclic order, it is not hard to see that the $ (\ell - a) $ vertices of $ Z $ not covered by $ E' $ belong to at least $ k - 1 + (\ell - a) $ edges of the cycle $ Z $, of which at least $ k - 1 + (\ell - a) - t $ also belong to $ Z' $.
        Since the number of edges of $ Z' $ not covered by $ E' $ is $ \ell - t - b $, this means that $ k - 1 + (\ell - a) - t \le \ell - t - b$ which is in turn equivalent to  $k + b - a \le 1 $.
        Note further that we can have equality only if all uncovered vertices are consecutive in $ Z $, which means also that the auxiliary edges of $ E' $ are consecutive, so if $ j' \ge 2 $ then they cannot all be of the same colour.
        Hence $ j' \ge 2 $ implies that either $ k + b - a \le 0 $ or $ m \ge 2 $, and we obtain the following.
        \begin{observation} \label{obs_cycles_conflict2}
            The total number of remaining vertex/colour choices for a conflict $ E $ containing $ E' $ is at most $ (j_1 - j') (\ell - k) + j_2 - \mathds{1}(j' \ge 2) $.
        \end{observation}
        We now use our observations to check the required conditions for the two types of conflict hypergraph.

        \subsubsection{The conflict hypergraph $ \mathcal{C} $}
        We define a conflict hypergraph $ \mathcal{C} $ for $ \mathcal{H}_1 $ by taking all conflicts $ E $ as described above for which $ j_2 = 0 $, and
        check that $ \mathcal{C} $ is $ (d, O(1), \varepsilon) $-bounded (for $ \varepsilon $ sufficiently small); the first condition (C\ref{cond_c1}) is clear since a cycle cannot be completely covered by two disjoint auxiliary edges.
        Both (C\ref{cond_c2}) and (C\ref{cond_c3}) follow from the cases $ j' = 1 $ and $ j' \ge 2 $ in Observation~\ref{obs_cycles_conflict2} respectively, since the conflict $ E $ is determined by these vertex/colour choices up to $ O(1) $ rearrangements, so $ \Delta_{j'}(\mathcal{C}^{(j_1)}) = O(d^{j_1 - j' - 2\varepsilon \mathds{1}(j' \ge 2)}) $ for all $ j' \in [j_1 - 1] $.

        \subsubsection{The conflict hypergraph $ \mathcal{D} $}
        We now define a conflict hypergraph $ \mathcal{D} $ for $ \mathcal{H} $ by taking all conflicts $ E $ with $ j_2 > 0 $; note that by the definition of $ Z' $, we have $ j_2 \ne 1 $ so (D\ref{cond_mixed_conflictsize}) holds.
        For (D\ref{cond_mixed_degree}), if we fix a single edge of the cycle $ f \in P $, then this fixes $ k $ vertices so by Observation~\ref{obs_cycles_conflict} we have $ O(n^{j_1 (\ell - k) + j_2}) \le d^{j_1 + \varepsilon^4} d_2^{j_2} $ choices for $ E $.
        Fixing a set $ E' $ of $ j' $ additional auxiliary edges from $ \mathcal{H}_1 $ which might be contained in conflicts with $ f $, the condition (D\ref{cond_mixed_codegree1}) follows immediately from Observation~\ref{obs_cycles_conflict2} in the case $ j' \ge 2 $.
        If instead $ j' = 1 $, note that $ E' $ cannot cover all vertices of $ f $, so fixing $ f $ fixes at least one further vertex, which is sufficient for (D\ref{cond_mixed_codegree1}).
        Finally for (D\ref{cond_mixed_codegree2}), observe that fixing a second edge $ g \in P $ fixes at least one further vertex of the cycle, which again suffices.

        \subsubsection{Conclusion}
		We may now conclude by applying Theorem~\ref{theorem_matching} to obtain a $ P $-perfect matching $ \mathcal{M} \subseteq \mathcal{H} $ containing no conflict from any of the conflict hypergraphs we have defined.
        As described, this yields a complete colouring $ c $ of $ E(K_n^k) $ with $ t_1 + t_2 $ colours in which every copy of $ C_{\ell}^k $ receives at least $ k + 1 $ distinct colours.
		\qed

    \subsection{The Erdős-Gyárfás function $ f(n,4,5) $}

        As mentioned in Section~\ref{section_applications_matching}, the following theorem was first proved by Bennett, Cushman, Dudek, and Prałat \cite{bennett2022erdhosgyarfas}, and the proof was already simplified massively by Joos and Mubayi \cite{joos2022ramsey}.
        We simplify it even further using Theorem~\ref{theorem_matching}.
        
        \begin{theorem}
			\label{theorem_kn_k4}
			$ r(K_n, K_4, 5) = \frac{5n}{6} + o(n) $.
		\end{theorem}
		
		\begin{proof}[Proof Sketch]
			We prove only the upper bound; see \cite{bennett2022erdhosgyarfas} for the lower bound.
			The proof is similar to the graph case of Theorem~\ref{theorem_cycles}, but complicated slightly by the use of a random construction of the hypergraph $ \mathcal{H} $, so we give a sketch of the key differences.
            Let $ t_2 \coloneqq n^{1 - \delta} $ for $ \delta > 0 $ sufficiently small, $ \rho \coloneqq n^{-\delta} $ and $ t_1 \coloneqq (1 + \rho)\frac{5n}{6} $. Let $ G \coloneqq K_n $ and construct the vertex set $ Q $ randomly as follows.
            Firstly take $ t_1 $ disjoint copies $ (V_{\alpha}')_{\alpha \in \mathcal{T}_1} $ of $ V(G) $, then delete each vertex independently at random with probability $ p \coloneqq \frac{\rho}{1 + \rho} $ to give sets $ (V_{\alpha})_{\alpha \in \mathcal{T}_1} $ of remaining vertices, and define
			$ Q \coloneqq \bigcup_{\alpha \in \mathcal{T}_1} V_{\alpha} $.
			Given a vertex $ v \in V(G) $ let $ v_{\alpha} $ denote the copy of $ v $ in $ V_{\alpha}' $, and define the hypergraph $ \mathcal{H}_1 $ by adding a hyperedge
			$$ \{ uv, vw, wu, u_{\alpha}, v_{\alpha}, w_{\alpha}, v_{\beta}, w_{\beta} \} $$
			for each triangle $ uvw $ in $ G $ and distinct $ \alpha, \beta \in \mathcal{T}_1 $ for which $ u_{\alpha}, v_{\alpha}, w_{\alpha}, v_{\beta}, w_{\beta} \in Q $ and $ u_{\beta} \not \in Q $.
            Choosing such a hyperedge corresponds to colouring the edges $ uv, uw $ by colour $ \alpha $ and $ vw $ by colour $ \beta $; when we do not need to specify the arrangement of the colours, we denote the hyperedge by $ (K, \alpha, \beta) $ for $ K = uvw $.
            A matching in $ \mathcal{H}_1 $ therefore yields a partial colouring in which every colour class consists of vertex-disjoint edges and two-edge paths such that, if the hyperedge above is chosen, then $ vw $ is a component in the colour class $ \beta $ and $ u $ is an isolated vertex in this class.
            By taking expected values and using a suitable concentration inequality (see \cite{joos2022ramsey} for details), it can be shown that $ \mathcal{H}_1 $ satisfies the required degree conditions with high probability, for some $ d $ which is $ \Theta(n^{3 - \delta}) $.
            We assume henceforth that we have fixed a specific choice of $ \mathcal{H}_1 $ which meets these conditions, as well as one further condition mentioned later.

            We define $ R $ to be a set of $ t_2 $ vertex disjoint copies $ (U_{\beta})_{\beta \in \mathcal{T}_2} $ of $ V(G) $, and the edges of $ \mathcal{H}_2 $ to be the set of pairs $ (uv, \beta) \in E(G) \times \mathcal{T}_2 $, each represented by the union of the edge $ \{ uv \} \subseteq P $ with the vertices $ \{ u_{\beta}, v_{\beta} \} \subseteq U_{\beta} $.
            It can be checked that this also satisfies the required degree conditions.

			Now suppose that $ Z $ is a copy of $ K_4 $ receiving at most four distinct colours in the colouring resulting from a matching $ \mathcal{M} \subseteq \mathcal{H} $.
            Note that crucially, by the construction of $ \mathcal{H}_1 $ and the fact that matchings of $ \mathcal{H}_2 $ yield proper colourings, some 4-cycle $ Z' \subseteq Z $ must receive exactly two distinct colours.
            Furthermore these colours must alternate, so we need consider only three cases depending on whether the two colours come from $ \mathcal{T}_1 $ or $ \mathcal{T}_2 $.
			
			Firstly, define a conflict hypergraph $ \mathcal{C} $ for $ \mathcal{H}_1 $ to be the colourings of 4-cycles in $ G $ comprising two monochromatic matchings of size 2; it can be checked that this is $ (d, O(1), \varepsilon) $-bounded \cite{joos2022ramsey}.
            Define also a $ (d, O(1), \varepsilon) $-simply-bounded conflict hypergraph $ \mathcal{D} $ for $ \mathcal{H}_2 $, with conflicts of size 4, to account for both colours coming from $ \mathcal{T}_2 $.
			
			Finally, define a conflict hypergraph $ \mathcal{E} $ for $ \mathcal{H} $ by taking conflicts of the form
			$$ \{ (K_x, \alpha_x, \alpha_x'), (K_y, \alpha_y, \alpha_y'), (xy, \beta), (x'y', \beta) \}, $$
			where
			\begin{itemize}
				\item $ K_x $ and $ K_y $ are vertex-disjoint triangles in $ G $;
				\item $ x, x' \in K_x $ and $ y, y' \in K_y $ (and are all distinct);
				\item the edges $ xx' $ and $ yy' $ both receive the same colour $ \alpha \in \mathcal{T}_1 $; 
				\item $ \alpha_x, \alpha_x', \alpha_y, \alpha_y' \in \mathcal{T}_1 $ and $ \beta \in \mathcal{T}_2 $.
			\end{itemize}
			
			Observe that this accounts for all of the ways in which the 4-cycle $ xx'y'y $ may receive the two colours $ \alpha $ and $ \beta $.
            It can be checked that $ \mathcal{E} $ is $ (d, O(1), \varepsilon) $-simply-bounded.
			
			We conclude by applying Theorem~\ref{theorem_matching} to obtain a matching containing none of the conflicts defined above, so in particular yielding a complete colouring of $ E(K_n) $ with $ t_1 + t_2 $ colours in which every copy of $ K_4 $ receives at least five distinct colours.
		\end{proof}

	\subsection{Coverings}
		\label{section_covering_proof}
		
		Here we prove Theorem~\ref{theorem_covering} as an application of Theorem~\ref{theorem_matching}, and then Theorem~\ref{theorem_covering_application} as a further application of Theorem~\ref{theorem_covering}; of course Theorem~\ref{theorem_covering_application} could also be proved directly from Theorem~\ref{theorem_matching}.
		
		\subsubsection{Proof of Theorem~\ref{theorem_covering}}
		
		As for the colouring results above we start by constructing a suitable hypergraph, then we construct conflict hypergraphs and check the required conditions.
		
		\textbf{Construction of $ \mathcal{H}_1 $ and $ \mathcal{H}_2 $:}
		Assume that $ \ell' $ and $ \varepsilon' $ are sufficiently large and small respectively, in terms of $ \ell, k, \varepsilon $.
		Define a new hypergraph $ \mathcal{H}' \coloneqq \mathcal{H}_1 \cup \mathcal{H}_2 $ as follows.
		We set
		$$ P \coloneqq R \coloneqq V(\mathcal{H}), \quad Q \coloneqq \emptyset, \quad E(\mathcal{H}_1) \coloneqq E(\mathcal{H}). $$
		Then take
		$$ E(\mathcal{H}_2) \coloneqq \left\{ v \cup E' : v \in P, E' \in {R \choose k - 1}, v \cup E' \in E(\mathcal{H}) \right\} $$
		to be a duplicate set of edges of $ \mathcal{H} $, but now with only a single vertex coming from $ P $ and the rest of the vertices coming from $ R $.
		
		We have $ p = k, q = 0, r = k - 1 $. Conditions (H\ref{cond_h_h1degree}) and (H\ref{cond_h_h1codegree}) follow immediately from the conditions for $ \mathcal{H} $, taking $ d $ to be the same.
		If $ v \in P $, then $ (1 - d^{-\varepsilon}) d \le d_{\mathcal{H}_2}(v) \le d $ and if $ v \in R $, then $ d_{\mathcal{H}_2}(v) \le (k - 1) d $ so (H\ref{cond_h_h2degree}) holds with $ \delta_{\mathcal{H}_2}(P) \ge (1 - d^{-\varepsilon}) d $.
		Given $ x \in P $ and $ v \in R $, we have $ d(x, v) \le d^{1 - \varepsilon} $ by the codegree condition for $ \mathcal{H} $, so (H\ref{cond_h_h2codegree}) is also satisfied, taking $ \varepsilon' \le \varepsilon / 2 $.
		
		Since we may associate edges in each of $ \mathcal{H}_1 $ and $ \mathcal{H}_2 $ with edges of $ \mathcal{H} $, a $ P $-perfect matching $ \mathcal{M}' \subseteq \mathcal{H}' $ as obtained from Theorem~\ref{theorem_matching} corresponds to a set of edges $ \mathcal{M} \subseteq \mathcal{H} $ such that every vertex $ v \in V(\mathcal{H}) $ is contained in at least one edge of $ \mathcal{M} $.
		Furthermore, a vertex may only appear multiple times if it appears once from $ P $ and once from $ R $, so at most $ (k - 1) d^{-\varepsilon^4}n $ vertices appear more than once in $ \mathcal{M} $, and no vertex appears more than twice.
		
		\textbf{Conflicts in $ \mathcal{H}_1 $:} Take $ \mathcal{C}' \coloneqq \mathcal{C} $, with the corresponding edges in $ \mathcal{H}_1 $. Then this immediately implies that $ \mathcal{C}' $ is $ (d, \ell', \varepsilon') $-bounded for any $ \ell' \ge \ell $ and $ \varepsilon' \le \varepsilon $.
		
		\textbf{Conflicts in $ \mathcal{H}' $:}
		Take $ \mathcal{D} $ to be the collection of all sets of edges $ C' \subseteq \mathcal{H}' $ such that the corresponding set of edges $ C \subseteq \mathcal{H} $ is a conflict in $ \mathcal{C} $, and at least one edge of $ C' $ comes from $ \mathcal{H}_2 $.
		Then (E\ref{cond_mixed_general_conflictsize}) is clear from (C\ref{cond_c1}). Given an edge $ e \in \mathcal{H}_2 $, there are at most $ \ell d^{j - 1} $ conflicts $ C \in \mathcal{C} $ of size $ j $ containing $ e $ (viewed as an edge of $ \mathcal{H} $) by (C\ref{cond_c2}), and given $ C $, there are at most $ {j - 1 \choose j_1} k^{j_2 - 1} $ corresponding $ C' \in \mathcal{D}^{(j_1, j_2)} $ for each pair $ j_1, j_2 $ with $ j_1 + j_2 = j $.
		Now given a vertex $ x \in P $, there are at most $ d $ such edges $ e \in \mathcal{H}_2 $ containing $ x $, so $ |\mathcal{D}_x^{(j_1, j_2)}| \le \ell' d^j \le d^{j_1 + \varepsilon'^4} d_2^{j_2} $, provided $ \ell' $ is chosen to be sufficiently large (noting that it may differ from $ \ell $), so (E\ref{cond_mixed_general_degree}) holds.
		In the case that $ j_2 = 1 $, we also obtain $ d_{\mathcal{D}^{(j_1, 1)}}(e) \le \ell' d^{j_1} $, as required for (E\ref{cond_mixed_general_degree_j21}).
		
		Similarly, any fixed set of $ j' $ edges in $ \mathcal{H}_1 $ are contained in at most $ d^{j - j' - \varepsilon} $ conflicts $ C \in \mathcal{C} $ of size $ j $ by (C\ref{cond_c3}), which gives (E\ref{cond_mixed_general_codegree1}) for suitable $ \varepsilon' $.
		The condition (E\ref{cond_mixed_general_codegree_j21}) follows similarly by fixing one further edge in $ \mathcal{H}_2 $.
		Finally, if we fix two vertices $ x, y \in P $, then there are at most $ d^2 $ edges $ e, f \in \mathcal{H}_2 $ with $ x \in e $ and $ y \in f $, and then there are at most $ d^{j - 2 - \varepsilon} $ conflicts $ C \in \mathcal{C} $ of size $ j $ containing both $ e $ and $ f $ (again by (C\ref{cond_c3})), so (E\ref{cond_mixed_general_codegree2}) follows.
		Hence $ \mathcal{D} $ is $ (d, \ell', \varepsilon', \varepsilon'^4) $-mixed-bounded.
		
		We finish by applying Theorem~\ref{theorem_matching}, and it is clear that $ \mathcal{M}' $ being $ \mathcal{C}' \cup \mathcal{D} $-free ensures that $ \mathcal{M} $ is $ \mathcal{C} $-free.
		\qed
		
	\subsubsection{Proof of Theorem~\ref{theorem_covering_application}}
	\label{section_covering_application_proof}
	
		Our strategy will be to apply Theorem~\ref{theorem_covering} to a suitable hypergraph with suitably defined conflicts. As such, let $ \mathcal{H} $ be the $ {s \choose t} $-graph on vertex set $ V(\mathcal{H}) = E(G) $ with edges $ {S \choose t} $ for each $ S \in \mathcal{K} $, and regard each such edge as corresponding to the $ s $-set $ S $.
		By assumption, $ \mathcal{H} $ satisfies (H\ref{cond_h_h1degree}) with $ d = (1 + m^{-\varepsilon})c m^{s - t} $ and $ \varepsilon' < \varepsilon $.
		For (H\ref{cond_h_h1codegree}), fix two distinct $ t $-sets $ T_1, T_2 \in {[m] \choose t} $, so $ |T_1 \cup T_2| \ge t + 1 $, meaning that the total number of $ s $-sets $ S \in {[m] \choose s} $ containing both $ T_1 $ and $ T_2 $ is at most
		$ {m \choose s - t - 1} \le m^{s - t - 1} \le d^{1 - \varepsilon'} $
		for $ \varepsilon' $ sufficiently small.

		Call a collection of $ j $ sets $ S_1, \ldots, S_j \in \mathcal{K} $, each of size $ s $, a \textit{bad $ j $-configuration} if it is a matching in $ \mathcal{H} $ (that is the intersection $ S_i \cap S_{i'} $ of any pair of the sets has size at most $ t - 1 $), and
		$$ |S_1 \cup \cdots \cup S_j| \le (s - t)j + t \text{.} $$
		Call such a configuration \textit{minimal} if it does not contain any bad $ j' $-configuration for $ 2 \le j' < j $.
		We may then take $ \mathcal{C} $ to be the set of all minimal bad $ j $-configurations (contained in $ \mathcal{K} $) for $ 3 \le j \le \ell $, noting that if two $ s $-sets span $ 2s - t $ points then they cannot be a matching, so we need not consider the case $ j = 2 $.
		Now it is clear that a $ \mathcal{C} $-free covering of $ \mathcal{H} $ as in Theorem~\ref{theorem_covering} yields the required collection $ \mathcal{S} $. Hence it suffices to check that $ \mathcal{C} $ is indeed $ (d, \ell', \varepsilon') $-bounded for $ \ell' $ and $ \varepsilon' $ sufficiently large and small respectively, in terms of $ s, t, \ell, \varepsilon $.
		
		Indeed, (C\ref{cond_c1}) holds for $ \ell' \ge \ell $.
		By definition, a conflict $ C $ of size $ j $ spans at most $ (s - t)j + t $ points of $ [m] $, and fixing one $ s $-set $ S \in C $ fixes $ s $ of these.
		This leaves at most $ (s - t)(j - 1) $ points to choose, so at most
		$ O(m^{(s - t)(j - 1)}) \le \ell' d^{j - 1} $
		choices for the rest of $ C $, for $ \ell' $ sufficiently large in terms of $ s, t, \ell $, which is sufficient for (C\ref{cond_c2}).
		For (C\ref{cond_c3}), suppose now that we fix $ j' $ sets $ S_1, \ldots, S_{j'} \in C $, for some $ 2 \le j' \le j - 1 $.
		Then by minimality of $ C $, we may assume that
		$ |S_1 \cup \cdots \cup S_{j'}| > (s - t)j' + t, $
		so there are at most $ (s - t)(j - j') - 1 $ points left to choose, giving at most
		$ O(m^{(s - t)(j - j') - 1}) \le d^{j - j' - \varepsilon'} $
		choices for $ C $, as required.
		Hence we may apply Theorem~\ref{theorem_covering} to give the desired conclusion.
		\qed

\end{document}